\documentclass{amsart}

\usepackage{amssymb}

\usepackage[pdftex]{graphicx,color}

\usepackage[cmtip,all]{xy}

\usepackage{url}
\usepackage{tikz}
\usepackage{tkz-graph}

\usepackage{algorithmic}
\algsetup{linenodelimiter=}

\newcommand{\UPTO}{\textbf{ up to }}
\newcommand{\DOWNTO}{\textbf{ down to }}

\copyrightinfo{2009}{David Perkinson}

\newtheorem{theorem}{Theorem}[section]
\newtheorem{lemma}[theorem]{Lemma}
\newtheorem{prop}[theorem]{Proposition}
\newtheorem{thm}[theorem]{Theorem}
\newtheorem{cor}[theorem]{Corollary}
\newtheorem{conj}[theorem]{Conjecture}

\theoremstyle{definition}
\newtheorem{definition}[theorem]{Definition}
\newtheorem{example}[theorem]{Example}
\newtheorem{notation}[theorem]{Notation}
\newtheorem{question}[theorem]{Question}

\newtheorem{algorithm}[theorem]{Algorithm}
\newenvironment{myalgorithm}{
    \begin{algorithm}
    \begin{algorithmic}[1]
}{
    \end{algorithmic}
    \end{algorithm}
}

\theoremstyle{remark}
\newtheorem{remark}[theorem]{Remark}

\numberwithin{equation}{section}

\newcommand{\N}{\mathbb{N}}
\newcommand{\Z}{\mathbb{Z}}
\newcommand{\C}{\mathbb{C}}

\newcommand{\R}{\mathbb{R}}

\newcommand{\tV}{\widetilde{V}}
\newcommand{\cmax}{c_{\mathrm{max}}}
\newcommand{\Dmax}{D_{\mathrm{max}}}
\DeclareMathOperator{\wt}{wt}
\DeclareMathOperator{\outdeg}{outdeg}
\DeclareMathOperator{\indeg}{indeg}
\DeclareMathOperator{\Hom}{Hom}
\DeclareMathOperator{\supp}{supp}
\DeclareMathOperator{\divisor}{div}
\DeclareMathOperator{\im}{im}

\DeclareMathOperator{\myspan}{Span}
\newcommand{\tD}{\widetilde{\Delta}}
\newcommand{\sand}{\mathcal{S}}

\newcommand{\Lap}{\mathcal{L}}
\newcommand{\tLap}{\widetilde{\mathcal{L}}}
\newcommand{\lt}{\mathrm{LT}}

\newcommand{\ba}{\begin{eqnarray*}}
\newcommand{\ea}{\end{eqnarray*}}
\newcommand{\be}{\begin{enumerate}}
\newcommand{\ee}{\end{enumerate}}

\begin{document}

\title[Algebraic Geometry of Sandpiles]{Primer for the Algebraic Geometry of Sandpiles}


\author{David Perkinson}
\address{Reed College, Portland OR 97202}
\email{davidp@reed.edu}

\author{Jacob Perlman}
\address{Department of Mathematics, University of Chicago, Chicago, Illinois 60637}
\email{perlman@math.uchicago.edu}

\author{John Wilmes}
\address{Department of Mathematics, University of Chicago, Chicago, Illinois 60637}
\email{wilmesj@math.uchicago.edu}


\date{\today}

\begin{abstract}
  The Abelian Sandpile Model (ASM) is a game played on a graph realizing the
  dynamics implicit in the discrete Laplacian matrix of the graph.  The purpose
  of this primer is to apply the theory of lattice ideals from algebraic
  geometry to the Laplacian matrix, drawing out connections with the ASM.  An
  extended summary of the ASM and of the required algebraic geometry is
  provided.  New results include a characterization of graphs whose Laplacian
  lattice ideals are complete intersection ideals; a new construction of
  arithmetically Gorenstein ideals; a generalization to directed multigraphs of
  a duality theorem between elements of the sandpile group of a graph and the
  graph's superstable configurations (parking functions); and a characterization
  of the top Betti number of the minimal free resolution of the Laplacian
  lattice ideal as the number of elements of the sandpile group of least degree.
  A characterization of all the Betti numbers is conjectured.
\end{abstract}

\maketitle


\section{Introduction} This is a primer on the algebraic geometry of sandpiles
based on lectures given by the first author in an undergraduate Topics in
Algebra course at Reed College in the fall of 2008 and on subsequent summer and
undergraduate thesis projects by the second and third authors.  It is assumed
that the reader has no background in algebraic geometry or the theory of
sandpiles but is willing to consult introductory outside sources such as~\cite{CLO}
and~\cite{Holroyd}.

The Abelian Sandpile Model (ASM) is a game in which one is allowed to stack
grains of sand on the vertices of a graph $G$.  If a vertex acquires too much
sand, a grain of sand will be fired to each neighboring vertex.  These vertices,
in turn, may become unstable, and an avalanche of vertex firings may ensue.  One
vertex is usually specified as a {\em sink}.  Its purpose is to absorb
sand fired into it, allowing avalanches caused by the addition of sand to
eventually come to a halt.  The ASM associates a group, the {\em sandpile
group}, to this sand-firing process.  The firing rule and the sandpile group are
intimately connected to the Laplacian of $G$.

In algebraic geometry, there is a way of associating a collection of polynomial
equations to an integer matrix.  These polynomials span the {\em lattice
ideal} corresponding to the matrix.  Our purpose is to apply the theory of
lattice ideals in the special case where the matrix in question is the Laplacian
matrix of a graph $G$, expressing the results in terms of sand on a graph.

There is another, more widely-known, connection between algebraic geometry and
sandpiles.  It comes from viewing a graph as a discrete version of a Riemann
surface (i.e., of an algebraic curve over $\C$).  As part of this connection,
there is a rich theory of divisors on graphs, including a version of the
Riemann-Roch theorem~\cite{Baker}.  In Sections 7 and 8, we see that this theory
is also relevant for our purposes.

We now give a summary of the paper by section.  Section~2 is an extended
outline of algebraic results associated with the Abelian Sandpile Model on a
graph.  What might be new here is a novel treatment of burning configurations
(Speer's script algorithm), an extension of the result expressing the
independence of the sandpile group from the choice of sink vertex, and the exposition
of the fact that an undirected planar graph and its dual have isomorphic
sandpile groups.

After a brief summary of the theory of lattice ideals in Section~3, our main
object of study---the toppling ideal of a graph---is introduced in Section~4.
The first paper on the algebraic geometry of sandpiles of which we are aware is
{\em Polynomial ideals for sandpiles and their {G}r\"obner bases}, by Cori,
Rossin, and Salvy \cite{CRS}. That paper defines the toppling ideal of an
undirected graph and computes a Gr\"obner basis for the ideal with respect to a
certain natural monomial ordering. Sections~4 and~5---building on results in the
undergraduate thesis of the second author~\cite{Perlman}---extend their work,
putting it in the context of lattice ideals and, in Theorem~\ref{thm: groebner
basis}, generalizing the Gr\"obner basis result to the case of directed
multigraphs.  The proof of Proposition~\ref{prop:generators}, giving generators
for the toppling ideal, is representative of the interplay between algebraic
geometry and sandpile theory. 

By Theorem~\ref{thm:sandpile lattice ideals}, any lattice ideal whose zero set
is finite is the lattice ideal corresponding to some directed multigraph.  In
that sense, the potential application of sandpile methods to lattice ideals is
quite broad.  As an application of algebraic geometry to the ASM,
Corollary~\ref{cor:duality} uses Gr\"obner bases to establish a duality between
elements of the sandpile group and superstable configurations ($G$-parking
functions).  The result is well-known for undirected graphs.  The proof given
here is the only one of which we know that works in the more general setting of a
directed multigraph.

Section 6 gives an explicit description of the zero set of the toppling ideal.
It is a generic orbit of a faithful representation of the sandpile group of the
graph.  The affine Hilbert function of the toppling ideal is defined in terms of
the sandpile group.  It is related to the Tutte polynomial of the graph by a
theorem of Merino~\cite{Merino}.  Proposition~\ref{prop:CB-yes} shows that the set of
zeros of the toppling ideal satisfies the Cayley-Bacharach property.

Section~7 summarizes the Riemann-Roch theory for graphs and includes results
obtained in the undergraduate thesis of the third author concerning the minimal
free resolution of the homogeneous toppling ideal of an undirected graph.  The
resolution is graded by the {\em class group} of the graph, closely related to
the sandpile group.  By a theorem of Hochster, the Betti numbers are determined
by the simplicial homology of complexes forming the supports of complete linear
systems on the graph.  By Theorem~\ref{thm:last betti number}, the top Betti
number counts the following structures on a graph: the elements of the sandpile
group of minimal degree, the maximal degree superstable configurations, the
maximal $G$-parking functions, the acyclic orientations with a unique fixed
source, and the non-special divisors.  Conjecture~\ref{conj:wilmes} suggests a
characterization all of the Betti numbers in terms of sandpile groups of graphs
associated with connected partitions (bonds) of the original graph.  For more on
resolutions of toppling ideals and a generalization of the Riemann-Roch theory
for graphs to certain monomial ideals, see the paper by Manjunath and
Sturmfels~\cite{Madhu}.

Finally, in Section 8, we characterize directed multigraphs whose homogeneous
toppling ideals are complete intersection ideals.  Further, we give a new method
of constructing directed multigraphs whose homogeneous toppling ideals are
arithmetically Gorenstein.
\medskip

\noindent{\sc Acknowledgments.}  We would like to express our thanks to students
who participated in the Topics in Algebra course at Reed College in fall 2008.
We would also like to acknowledge the contribution of Luis David Garcia.  He
suggested that we look at~\cite{MT} in the context of sandpile ideals, leading
us to Theorems~\ref{ci implies wired} and~\ref{wired implies ci}. We thank Bernd
Sturmfels and Madhusudan Manjunath for their encouragement and comments, and we
thank Collin Perkinson for comments on the exposition.  The first author would
like to thank Tony Geramita and Lorenzo Robbiano for introducing him to the
geometry of finite sets of points.

This work could not have been done without the help of the mathematical software
system Sage~\cite{sage}.  Interested readers may want to consult the Thematic
Tutorial in the Help/Documentation section of the Sage homepage,
\url{sagemath.org}.  It contains an introduction to the ASM with computational
examples.  For visualization of the ASM, the reader is referred to Bryan
Head's Google Summer of Code project, available at
\url{www.reed.edu/~davidp/sand/program}. 

\section{Sandpiles}\label{sandpiles}
In~this section we summarize the basic theory of sandpile groups.  Many results
are stated without proof.  The reader is referred to~\cite{Holroyd}
and~\cite{Perkinson} for a thorough introduction to the subject.

\subsection{Graph theory} Let $G=(V,E)$ be a directed multigraph with a finite
set of vertices $V$ and of directed edges $E$.  For
$e=(u,v)\in E\subseteq V\times V$, we write $e^{-}:=u$ and $e^{+}:=v$ for the {\em tail} and {\em
head} of $e$, respectively. If $e^{-}=e^{+}$, the edge is a {\em loop}.  These
are allowed but do not add much to the theory.  By
``multigraph'' we will mean that there is a {\em weight function},
\[
\wt\colon V\times V\to\N,
\]
such that $\wt(u,v)>0$ if and only if $(u,v)\in E$. One may think of an
edge $e=(u,v)$ of $\wt(e)$ as $\wt(e)$ edges connecting $u$ to $v$.
For $v\in V$, 
\begin{eqnarray*}
  \outdeg(v)&:=& \sum_{e\in E:e^{-}=v}\wt(e)\\
  \indeg(v)&:=& \sum_{e\in E:e^{+}=v}\wt(e).
\end{eqnarray*}
The graph $G$ is {\em undirected} if $\mbox{wt}(u,v)=\mbox{wt}(v,u)$ for all
$u,v\in V$, and it is {\em unweighted} if the weights of all of its edges are
$1$.  If $G$ is undirected, we use the notation $\deg(v):=\outdeg(v)=\indeg(v)$.

A vertex $u$ is {\em accessible} from a vertex $v$ if there is a directed path
beginning at $u$ and ending at $v$.  A vertex $s$ is {\em globally accessible}
if it is accessible from all vertices of $G$.   Throughout this primer, {\bf we
will only consider graphs having at least one globally accessible vertex}.  In
particular, undirected graphs are assumed to be connected.

\begin{definition}
  A {\em sandpile graph} is a triple $(V,E,s)$ consisting of a finite, directed
  multigraph $(V,E)$ with a globally accessible vertex $s$.  The vertex $s$ is
  called the {\em sink} of the sandpile graph.  If, in addition, $s$ has
  outdegree $0$, it is called an {\em absolute sink}. The nonsink vertices are
  denoted $\tV:=V\setminus\{s\}$.  
\end{definition}
If $G=(V,E,s)$ is a sandpile graph, we will also refer to the graph $(V,E)$
as~$G$.  Note that the sink of a sandpile graph need not be absolute;
however, for much of what we say, one could safely remove outgoing edges from
the sink without changing the theory. 
\begin{example}
  Figure~\ref{fig:sandpile graph} depicts a sandpile graph $G$.  Edges
  $(v_1,v_2)$, $(v_2,v_1)$, $(v_2,s)$, $(v_3,s)$, and $(s,v_3)$ are directed
  edges with weights $1,2,1,1,5$, respectively; $\{v_1,v_3\}$ is an undirected
  edge of weight $3$; and $\{v_2,v_3\}$ is an undirected, unweighted edge.
  Although $s$ is the sink of the sandpile graph,~$\outdeg(s)=5$.
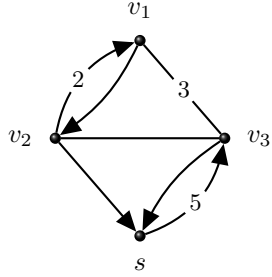
\begin{figure}[ht] 
\begin{tikzpicture}[scale=1.3]
\SetVertexMath
\GraphInit[vstyle=Art]
\SetUpVertex[MinSize=3pt]
\SetVertexLabel
\tikzset{VertexStyle/.style = {%
shape = circle,
shading = ball,
ball color = black,
inner sep = 1.5pt
}}
\SetUpEdge[color=black]
\Vertex[LabelOut,Lpos=90, Ldist=.1cm,x=0,y=1]{v_1}
\Vertex[LabelOut,Lpos=180, Ldist=.1cm,x=-0.866,y=0]{v_2}
\Vertex[LabelOut,Lpos=0, Ldist=.1cm,x=+0.866,y=0]{v_3}
\Vertex[LabelOut,Lpos=270, Ldist=.1cm,x=0,y=-1]{s}
\Edge[](v_1)(v_3)
\fill[color=white] (0.45,0.5) circle (0.15cm);
\draw (0.45,0.5) node{{\small 3}};
\Edge[](v_2)(v_3)
\Edge[style={-triangle 45,bend left=15}](v_1)(v_2)
\Edge[style={-triangle 45,bend left=30}](v_2)(v_1)
\fill[color=white] (-0.63,0.6) circle (0.15cm);
\draw (-0.63,0.6) node{{\small 2}};
\Edge[style={-triangle 45,bend right=15}](v_3)(s)
\Edge[style={-triangle 45}](v_2)(s)
\Edge[style={-triangle 45,bend right=30}](s)(v_3)
\fill[color=white] (0.57,-0.66) circle (0.15cm);
\draw (0.57,-0.66) node{{\small 5}};
\end{tikzpicture}
\caption{Sandpile graph $G$ with sink $s$.}\label{fig:sandpile graph}
\end{figure}
\end{example}

For any finite set $X$, let 
\[
\Z X=\{\textstyle\sum_{x\in X}a_x\,x:a_x\in\Z\mbox{ for all $x\in X$}\}
\]
be the free Abelian group on $X$.  Restricting to nonnegative coefficients gives
$\N X$. 
\begin{notation}
For $a,b\in\Z X$, we define $\deg(a)=\sum_{x\in X}a_x$ and $a\geq b$ if
$a_x\geq 
b_x$ for all $x\in X$.  We say $a$ is {\em nonnegative} if $a\geq0$. The {\em support} of $a$ is 
\[
\supp(a)=\{x\in X:a_x\neq0\}.
\]
Similar notation is used for integer vectors.
\end{notation}

Let $G=(V,E,s)$ be a sandpile graph.
\begin{definition}
  The {\em (full) Laplacian} of $G$ is the mapping of groups $\Delta\colon\Z
  V\to\Z V$ given on vertices $v$ by
  \[
  \Delta(v):=\outdeg(v)\,v-\sum_{u\in V}\wt(v,u)\,u.
  \]
  The {\em reduced Laplacian} of $G$ is the mapping of groups
  $\tD\colon\Z\tV\to\Z\tV$ given on nonsink vertices $v$ by
  \[
  \tD(v):=\outdeg(v)\,v-\sum_{u\in\tV}\wt(v,u)\,u,
  \]
  summing this time only over $\tV$.
\end{definition}
The Laplacian just defined is dual to the Laplacian one often sees in the
literature.  Define $L\colon \Z^{V}\to\Z^{V}$ by
\[
L\phi(v):=\sum_{u\in V}\wt(v,u)(\phi(v)-\phi(u))
\]
for a function $\phi\in\Z^{V}$ and vertex $v$. Say $V=\{v_1,\dots,v_{n+1}\}$,
and define the diagonal matrix
$D=\mbox{diag}(\outdeg(v_1),\dots,\outdeg(v_{n+1}))$.   Let $A$ be the {\em
adjacency matrix}, $A$, given by $A_{ij}=\wt(v_i,v_j)$.  Fixing an ordering
$v_1,\dots,v_{n+1}$ of the vertices identifies~ $\Z^{V}$ with $\Z^{n+1}$ and
identifies $L$ with the $(n+1)\times(n+1)$ matrix
\[
L=D-A.
\]
The matrix for our Laplacian $\Delta$ of $G$ is the transpose of $L$.

A {\em spanning tree directed into $s$} is a subgraph $T$ of $G$ with the
properties: (1)~$T$ contains all of the vertices of $G$, (2) the weight
of each edge in $T$ is the same as its weight as an edge of $G$, (3) for each
vertex, there is a directed path in $T$ to $s$, (4) for each vertex $v\neq s$,
there is exactly one edge of $T$ whose tail is $v$, and (5) the outdegree of $s$
is $0$.  If $T$ is a spanning tree directed into $s$, then its weight, denoted
$\wt(T)$, is the product of the weights of its edges.  The following is a basic
theorem in graph theory.

\begin{thm}[Matrix-Tree]
  The determinant of the reduced Laplacian of $G$ is the sum of the weights of
  all its directed spanning trees into the sink.
\end{thm}

It will occasionally be useful to consider a more restricted class of graphs.
\begin{definition}
  A directed multigraph $G=(V,E)$ is {\em Eulerian} if each of its vertices is globally
  accessible and $\indeg(v)=\outdeg(v)$ for all $v\in V$.  
\end{definition}

Every undirected graph is Eulerian.  The condition that $\indeg(v)=\outdeg(v)$
for all vertices $v$ is equivalent to having $\vec{1}\in\ker\Delta$.

\subsection{The Sandpile Group}

Let $G=(V,E,s)$ be a sandpile graph with nonsink vertices $\tV$.
\begin{definition}
  A {\em (sandpile) configuration} on $G$ is an element of
  $\Z\tV$.  A configuration $c=\sum_{v\in\tV}c_v v$ is {\em stable} at a vertex
  $v\in\tV$ if $c_v<\outdeg(v)$.  Otherwise, it is {\em unstable}.  A
  configuration is {\em stable} if it is stable at each $v\in\tV$.
\end{definition}
As the name suggests, we think of a configuration $c$ as a pile of sand on the
nonsink vertices of $G$ having $c_v$ grains of sand at vertex $v$.  Sand can be
redistributed on the graph by vertex {\em firings (or topplings)}.  Firing
$v\in\tV$ in configuration $c$ gives the new configuration,
\begin{align*}
  \tilde{c}&=c-\outdeg(v)\,v+\sum_{u\in\tV}\wt(v,u)\,u\\
  &=c-\tD\,v.
\end{align*}
When $v$ fires, we imagine $\wt(e)$ grains of sand traveling along each edge $e$
emanating from $v$ and being deposited at $e^{+}$.  If $e^{+}=s$, then sand sent
along $e$ disappears down the sink.  If~$c$ is unstable at $v$, we say that
firing $v$ is {\em legal}.  The sequence of nonsink vertices $u_1,\dots,u_k$ is
a {\em legal firing sequence} for a configuration $c$ if it is legal to fire
$u_1$ and then it is legal to fire each subsequent $u_i$ from the configuration
obtained by firing $u_1,\dots,u_{i-1}$.  The configuration resulting from
applying a legal firing sequence to $c$ is the configuration
$\tilde{c}=c-\tD\,\sigma$ where $\sigma\in\Z\tV$ is such that $\sigma_v$ is the
number of times vertex $v$ appears in the sequence.  We write
\[
c\stackrel{\sigma}{\longrightarrow}c-\tD\,\sigma.
\]
In general, we write $c\to\tilde{c}$ if $\tilde{c}$ is the result of applying a
legal firing sequence to $c$.  In this case, since the reduced Laplacian is
invertible (by the Matrix-Tree theorem, for instance), there exists a unique
$\sigma\in\Z\tV$ such that
$\tilde{c}=c-\tD\,\sigma$.  This $\sigma$ is called the {\em firing script} or
{\em firing vector} for $c\to\tilde{c}$.

We have the following existence and uniqueness theorem.
\begin{thm} Let $c$ be a sandpile configuration.
  \begin{enumerate}
    \item There exists a stable configuration $\tilde{c}$ such that
      $c\to\tilde{c}$.
    \item Suppose $c\to\tilde{c}$ with script $\sigma$ and $c\to\tilde{c}'$ with
      script $\sigma'$. Then if $\tilde{c}$ is stable, $\sigma'\geq\sigma$.  If
   $\tilde{c}$ and $\tilde{c}'$ are both stable, then $\tilde{c}=\tilde{c}'$.
  \end{enumerate}
\end{thm}

\begin{definition}  Let $c$ be a configuration on $G$. 
  The {\em stabilization} of a configuration~$c$, denoted $c^{\circ}$, is the
  unique stable configuration $\tilde{c}$ such that $c\to\tilde{c}$.  
\end{definition}

Let $\mathcal{M}$ denote the set of nonnegative stable configurations on $G$.  Then
$\mathcal{M}$ is a commutative monoid under {\em stable addition}
\[
a\circledast b :=(a+b)^{\circ}.
\]
Thus, stable addition is vector addition in $\N\tV$ followed by
stabilization.  The identity is the zero configuration.

\begin{definition} A configuration $c$ is {\em accessible} if for each
  configuration $a$, there exists a configuration $b$ such that $a+b\to c$.  A
  configuration $c$ is {\em recurrent} if it is nonnegative, accessible, and
  stable.
\end{definition}

\begin{definition}
  The {\em maximal stable configuration} on $G$ is the configuration
\[
\cmax = \sum_{v\in\tV}(\outdeg(v) - 1)v.
\]
\end{definition}

\begin{prop}\label{prop:max-stable}
A configuration $c$ is recurrent if and only if there exists a 
configuration $a\geq0$ such that
\[
c = a\circledast\cmax.
\]
\end{prop}

It is not hard to see that the recurrent elements form a semigroup.  In fact,
they form a group.
\begin{thm}
  The collection of recurrent configurations of $G$ forms a group under stable
  addition.
\end{thm}

\begin{definition}
  The group of recurrent configurations of a sandpile graph $G$ is called the
  {\em sandpile group} of $G$ and denoted by $\sand(G)$.
\end{definition}

By Proposition~\ref{prop:max-stable}, the sandpile group can be found by a
systematically adding sand to $\cmax$ and stabilizing.  Considering a graph
consisting of otherwise unconnected vertices connected into a common sink by
edges of various weights, one sees that every finite Abelian group is the
sandpile group for some graph.
\begin{example}\label{example:sandpile graph}
  The elements of the sandpile group for the sandpile graph in
  Figure~\ref{fig:sandpile graph} are listed below using the notation
  $(c_1,c_2,c_3):=c_1v_1+c_2v_2+c_3v_3$:
 \[
 \begin{array}{ccccccc} 
   (3, 3, 4)& (3, 3, 3)& (3, 2, 4)& (2, 3, 4)& (3, 3, 2)& (3, 2, 3)& (2, 3, 3)\\
   (3, 1, 4)& (2, 2, 4)& (1, 3, 4)& (3, 2, 2)& (2, 2, 3)& (1, 3, 3)& (3, 0, 4)\\
   (2, 1, 4)& (1, 2, 4)& (0, 3, 4)& (1, 2, 3)& (0, 3, 3)& (2, 0, 4)& (1, 1, 4)
 \end{array}
 \]
\end{example}
Although the zero configuration is the identity for $\mathcal{M}$, it is seldom
the identity for $\mathcal{S}(G)$.  The following is an easy exercise.
\begin{prop}
The following are equivalent:
  \begin{enumerate}
    \item the zero-configuration $\vec{0}$ is recurrent;
    \item every stable configuration is recurrent;
    \item every directed cycle of $G$ passes through the sink vertex.
  \end{enumerate}
\end{prop}

We now give another description of the sandpile group.

\begin{definition}
  The {\em Laplacian lattice}, $\Lap\subset\Z V$, is the image of $\Delta$.
  The reduced Laplacian lattice, $\tLap\subset\Z \tV$, is the image of~$\tD$.  The {\em critical group} for $G$ is
  \[
  \mathcal{C}(G) = \Z\tV/\tLap.
  \]
\end{definition}

\begin{thm}\label{thm:main isomorphism}
  There is an isomorphism of Abelian groups
  \begin{eqnarray*}
  \mathcal{S}(G)&\to&\mathcal{C}(G)\\
  c&\mapsto&c+\tLap.
\end{eqnarray*}
\end{thm}
Thus, each element of $\Z\tV$ is equivalent to a unique recurrent element modulo
the reduced Laplacian lattice.  The identity of the sandpile group is the
recurrent configuration in $\tLap$.  It can be calculated as
\[
\eta=((\cmax-(2\cmax)^{\circ})+\cmax)^{\circ}.
\]
Note that $\eta=0\bmod\tLap$, and since $\cmax-(2\cmax)^{\circ}\geq0$,
Proposition~\ref{prop:max-stable} guarantees that $\eta$ is recurrent.
\begin{example}\label{example:smith normal form}
  The reduced Laplacian of the sandpile graph in Figure~\ref{fig:sandpile graph}
  is
  \[
   \tD=
   \left(
   \begin{array}{rrr}
     4&-2&-3\\
     -1&4&-1\\
     -3&-1&5
   \end{array}
   \right).
  \]
  The Smith normal form of $\tD$ is $\mbox{diag}(1,1,21)$.  Hence,
  $\sand(G)\approx\Z/21\Z$. The identity is $(3,1,4)$, computed as follows:
  \begin{align*}
    (\cmax-(2\cmax)^{\circ})+\cmax&=((3,3,4)-(6,6,8)^{\circ})+(3,3,4)\\
    &=((3,3,4)-(2,0,4))+(3,3,4)\\
    &=(4,6,4)\rightsquigarrow (3,1,4).
  \end{align*}
\end{example}
As a consequence of the Matrix-Tree theorem, we have the following.
\begin{cor}
  The order of $\sand(G)$ is the sum of the weights of $G$'s directed spanning
  trees into $s$.
\end{cor}
\begin{remark}
Babai \cite{Babai} has noted another characterization of the sandpile group: it
is the principal semi-ideal in $\mathcal{M}$ generated by $\cmax$, which turns
out to be the intersection of all the semi-ideals of $\mathcal{M}$.
\end{remark}

\begin{remark} 
In the literature, a sandpile configuration is often taken to be an element of
$\Z^{\tV}$.  We prefer to work in the dual group $\Z\tV=\Hom(\Z^{\tV},\Z)$ so
that the functor that takes a sandpile graph to its sandpile group is covariant.
Suppose that $G=(V,E,s)$ and $G'=(V',E',s')$ are sandpile graphs with reduced
Laplacian lattices $\tLap$ and $\tLap'$, respectively.  Let $\Psi:G'\to G$ be a
mapping of graphs that maps~$s'$ to~$s$.  Applying $\mathrm{hom_{\Z}}(\ \cdot\
,\Z)$ to the natural induced map $\Z^{V}\to\Z^{V'}$ yields $\Psi_*\colon\Z
V'\to\Z V$.  If $\Psi(\tLap')\subseteq\Lap$, there is an induced mapping of
sandpile groups.  This condition would seem to define a reasonable set of
morphisms, then, for a category of sandpile groups.  For work on the category
theory of sandpile groups, see~\cite{Reiner} and ~\cite{Treumann}.  For the
notion of a {\em harmonic morphism} of graphs, see~\cite{Baker2}.
\end{remark}

\subsection{Superstables}  Let $c=u+v$ be a configuration on the (unweighted,
undirected) sandpile graph in
Figure~\ref{fig:not superstable} with sink $s$.
\begin{figure}[ht] 
\begin{tikzpicture}
\SetVertexMath
\GraphInit[vstyle=Art]
\SetUpVertex[MinSize=3pt]
\SetVertexLabel
\tikzset{VertexStyle/.style = {%
shape = circle,
shading = ball,
ball color = black,
inner sep = 1.3pt
}}
\SetUpEdge[color=black]
\Vertex[LabelOut,Lpos=180,
Ldist=.1cm,x=0,y=0]{u}
\Vertex[LabelOut,Lpos=0,
Ldist=.1cm,x=2,y=0]{v}
\Vertex[LabelOut,Lpos=270,
Ldist=.2cm,x=1,y=-1]{s}
\Edges(u,v,s,u)
\end{tikzpicture}
\caption{Graph $G$.}\label{fig:not superstable}
\end{figure}
The vertices $u$ and $v$ are both stable in $c$, so there are no legal vertex
firings: firing either vertex would result in a negative amount of sand on a
vertex.  However, firing both vertices simultaneously results in a nonnegative
configuration, the zero configuration.  Each nonsink vertex loses two
grains of sand, but each also gains a grain from the other.
\begin{definition} Let $c$ be a configuration on the sandpile graph $G=(V,E,s)$.
  A {\em script-firing}, also called a {\em cluster-} or {\em multiset-firing},
  with {\em (firing) script} $\sigma\in\N\tV$ is the operation that replaces $c$
  with $c-\tD\,\sigma$.  The script-firing is {\em legal} if
  $\sigma\gneq0$ and $(c-\tD\,\sigma)_v\geq0$ for each $v\in\supp(\sigma)$.
  Thus, if $c\geq0$, the script-firing with script $\sigma\gneq0$ is legal if
  and only if $c-\tD\,\sigma\geq0$.
 
  A configuration $c$ is {\em superstable} if $c$ is nonnegative and has no
  legal script-firings.
\end{definition}
The idea of a {\em $G$-parking function} is essentially the same as that of a
superstable configuration:
\begin{definition}  Let $G=(V,E,s)$ be a sandpile graph.
  A {\em $G$-parking function}~\cite{Postnikov} (with respect to $s$) is a
  function $f\colon V\to\Z$ such that there exists a superstable configuration $c$ on
  $G$ with the property that $f(v)=c_v$ for $v\in\tV$ and $f(s)=-1$.
\end{definition}

An {\em acyclic orientation} of an undirected graph $G$ is a choice of
orientation for each edge of $G$ such that the resulting directed graph has no
directed cycles.  A vertex $v$ is a {\em source} for an acyclic orientation if
all the edges incident on $v$ are directed away from $v$.  If $\mathcal{O}$ is
an acyclic orientation and $v\in V$, then $\indeg_{\mathcal{O}}(v)$ denotes the
indegree of $v$ for the directed graph corresponding to $\mathcal{O}$.
\begin{thm}[\cite{Benson}]\label{thm:acyclic orientations}
  Let $G=(V,E,s)$ be an undirected sandpile graph. Then there is a bijection
  between the set of acyclic orientations of $G$ with unique source $s$ and the
  set of superstable configurations of $G$ of highest degree.  If $\mathcal{O}$
  is an acyclic orientation, the corresponding maximal superstable configuration
  is given by
  \[
  \sum_{v\in\tV}(\indeg_{\mathcal{O}}(v)-1)\,v.
  \]
\end{thm}
For an extension of the previous theorem from maximal superstable configurations
to all superstable configurations (and a connection with hyperplane arragements), see~\cite{Hopkins}.

\subsection{Burning configurations}\label{subsect:burning}
Speer's script algorithm \cite{Speer} generalizes the burning algorithm of Dhar,
testing whether a configuration is recurrent.  We present a variation on Speer's
algorithm using burning configurations.     
\begin{definition}
A configuration $b$ is a \emph{burning configuration} if it has the
following three properties:
\begin{enumerate} 
  \item $b\in \tLap$,
  \item $b\geq 0$,
  \item for all $v \in \tV$, there exists a path to $v$
    from some element of $\mathrm{supp}(b)$.
\end{enumerate}
If $b$ is a burning configuration, we call $\sigma_{b} =
(\tD)^{-1}b$
the {\em script} or the {\em firing vector} for $b$.
\end{definition}

\begin{thm}[\cite{Perkinson}]\label{thm:speer}
Let $b$ be the burning configuration with script $\sigma_b$.
\begin{enumerate}
  \item $(kb)^{\circ}$ is the identity configuration for $k\gg0$.
  \item A configuration $c$ is recurrent if and only if the stabilization of
    $c+b$ is $c$.
  \item A configuration $c$ is recurrent if and only if the firing vector
    for the stabilization of $b+c$ is~$\sigma_b$.
  \item $\sigma_b\geq\vec{1}$.
  \item If $c$ is a configuration and $\tau$ is the firing vector for
    the stabilization of $c+b$, then $\tau\leq\sigma_b$.
\end{enumerate}
\end{thm}
Thus, a configuration $c$ is in the sandpile group if and only if adding a
burning configuration to $c$ and stabilizing returns $c$, or if, equivalently,
the firing script for the stabilization is equal to the burning script.  For the
case of an undirected graph, as we see in the following theorem, one may take
$\vec{1}$ as the firing script.  Adding the burning configuration to a
configuration $c$ in that case can be thought of as placing~$c$ on the graph,
then firing the sink vertex.  Checking whether each vertex fires
exactly once in the subsequent stabilization is known as {\em Dhar's algorithm}.

\begin{thm}[\cite{Speer},\cite{Perkinson}]\label{thm:speer existence}
  There exists a unique burning configuration $b$ with script
  $\sigma_b=\tD^{-1}b$ having the following property:
   if $\sigma_{b'}$ is the script for a burning configuration
    $b'$, then $\sigma_{b'}\geq\sigma_{b}$.
  For this $b$, we have:
  \begin{enumerate}
    \item For all $v\in\tV$, $b_v<\outdeg(v)$ unless $v$ is a source, i.e.,
      unless $\mathrm{indeg}(v)=0$, in which case $b_v=\outdeg(v)$.  Thus, $b$ is
      stable unless $G$ has a source, and in any case, $b_v\leq \outdeg(v)$ for
      all~$v$.
    \item $\sigma_b\geq\vec{1}$ with equality if and only if $G$ has no ``selfish''
      vertices, i.e., no vertex $v\in\tV$ with
      $\mathrm{indeg}(v)>\outdeg(v)$.
  \end{enumerate}
  We call this $b$ the {\em minimal burning configuration} and its script,
  $\sigma_b$, the {\em minimal burning script}.
\end{thm}
\begin{remark}
  To compute the minimal burning configuration, start with $b$ equal to the sum
  of the columns of $\tD$.  If $b\geq0$, stop.  Otherwise, if $b_v<0$ for some
  $v\in\tV$, replace $b$ by $b+\tD(v)$. Repeat until $b\geq0$. 
\end{remark}
\begin{example}  We would like to compute the minimal burning configuration and
  corresponding script for the sandpile graph $G$ in Figure~\ref{fig:sandpile
  graph}.  Continuing Example~\ref{example:smith normal form}, the sum of
  the columns of $\tD$ is $(-1,2,1)^t$.  Since the first entry of the sum is
  negative, add in the first column of $\tD$ to get $(3,1,-2)^t$.  Since
  the third entry is now negative, add in the third column of $\tD$ to get
  $(0,0,3)$.  Thus, the minimal burning configuration is $b=(0,0,3)$, and the
  burning script is $\sigma_b=(2,1,2)$, recording the columns of $\tD$ used to
  obtain~$b$.
\end{example}
\subsection{Some isomorphisms.}  
\subsubsection{Choice of sink vertex.} Lemma~4.12 of~\cite{Holroyd} states that
for Eulerian graphs, the sandpile group is, up to isomorphism, independent of the
choice of sink.  Here, we present a generalization of that result.

Let $G=(V,E,s)$ be a sandpile graph. Recall that
$\mathcal{C}(G):=\Z\tV/\tLap$ is the critical group of $G$, isomorphic to
the sandpile group, $\sand(G)$, by Theorem~\ref{thm:main isomorphism}.  Let
\[
\Z V_{0}:=\{c\in\Z V: \deg(c)=0\}.
\]
Since the image of the Laplacian $\Delta$ is contained in $\Z v_0$, we may
define the mapping $\Delta_0:\Z V\to\Z V_0$ by $\Delta_0(c):=\Delta(c)$ for all $c\in\Z V$.

\begin{prop}[\cite{Perkinson}]\label{prop:sink}\ 
  \begin{enumerate}
    \item\label{sink1} There is a commutative diagram with exact rows
\[
\xymatrix{
0\ar[r]&\Z\tV\ar[r]^{\tD}\ar[d]_{\iota}&\Z\tV\ar[r]\ar[d]_{\varepsilon}
&\sand(G)\ar[r]\ar[d]&0\\
0\ar[r]&\Z V/\ker\Delta\ar[r]^{\Delta_0}&\Z V_0\ar[r]&\mathcal{C}(G)\ar[r]&0.
}
\]
where $\iota(v):=v+\ker\Delta$ and $\varepsilon=v-s$ for all $v\in\tV$.
\item\label{sink2} For each $v\in V$, let $\tau_v$ be the sum of the weights of all spanning trees
  directed into $v$, let $d=\gcd\{\tau_u:u\in V\}$, and let
  $\tilde{\tau}_v:=\tau_v/d$.  Define $\tilde{\tau}:=\sum_v\tilde{\tau}_vv\in\Z
  V$.  Then
  \[
  \ker{\Delta}=\myspan_{\Z}\{\tilde{\tau}\}.
  \]
  
\item\label{sink3} There is a short exact sequence
  \[
0\longrightarrow\Z/\tilde{\tau}_s\Z\longrightarrow\mathcal{C}(G)
\longrightarrow\Z V_0/\Lap\longrightarrow0.
  \]
  \end{enumerate}
\end{prop}
  \begin{cor}
    If $G$ is an Eulerian graph (in particular, if $G$ is undirected), then the
    sandpile group for $G$ is independent of the choice of sink vertex.
  \end{cor}
  \begin{proof}
    Suppose $G$ is Eulerian.  Then each vertex is globally accessible.  So it
    makes sense to talk about the sandpile group of $G$ with respect to any of
    its vertices.  Since $\indeg(v)=\outdeg(v)$ for all $v\in V$, we have
    that $\vec{1}\in\ker\Delta$.  It follows from Proposition~\ref{prop:sink}
    (\ref{sink2}) that $\tilde{\tau}_v=1$ for all $v$.  Fix a vertex $s$ and
    consider the sandpile group of $G$ with respect to $s$.  It is isomorphic to
    the critical group (with respect to $s$), and hence isomorphic to $\Z
    V_0/\Lap$ by Proposition~\ref{prop:sink} (\ref{sink3}).  However, $\Z V_0/\Lap$
    does not depend on the choice of a sink.
  \end{proof}

\subsubsection{Planar duality.}  Let $G=(V,E)$ be an undirected graph.  Fix an
orientation~$\mathcal{O}$ of the edges of $G$.  Thus, for each $\{u,v\}\in E$
we have that either $(u,v)$ or $(v,u)$ is in $\mathcal{O}$, but not both.  Let
$e=\{u,v\}\in E$, and suppose that $(u,v)\in\mathcal{O}$.  In the free abelian
group $\Z E$, we identify $(u,v)$ with $e$ and $(v,u)$ with $-e$.  We also
define~$e^{-}:=u$ and $e^{+}:=v$. 

The {\em (integral) cycle space}, $\mathcal{C}=\mathcal{C}_G\subseteq\Z E$, is
the $\Z$-span of the cycles of $G$.
\begin{example}
  Let $G$ be the (undirected) triangle with edges oriented as in
  Figure~\ref{fig:oriented triangle}.
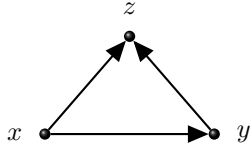
\begin{figure}[ht] 
\begin{tikzpicture}[scale=1.3]
\SetVertexMath
\GraphInit[vstyle=Art]
\SetUpVertex[MinSize=3pt]
\SetVertexLabel
\tikzset{VertexStyle/.style = {%
shape = circle,
shading = ball,
ball color = black,
inner sep = 1.5pt
}}
\SetUpEdge[color=black]
\Vertex[LabelOut,Lpos=90, Ldist=.1cm,x=0,y=1]{z}
\Vertex[LabelOut,Lpos=180, Ldist=.1cm,x=-0.866,y=0]{x}
\Vertex[LabelOut,Lpos=0, Ldist=.1cm,x=+0.866,y=0]{y}
\Edge[style={-triangle 45}](x)(y)
\Edge[style={-triangle 45}](x)(z)
\Edge[style={-triangle 45}](y)(z)
\end{tikzpicture}
\caption{A triangle with oriented edges.}\label{fig:oriented triangle}
\end{figure}
 The cycle space for $G$ is the $\Z$-span of the cycle $(x,y)+(y,z)-(x,z)$.
\end{example}

For each $U\subseteq V$, define the corresponding {\em cut-set}, $c^*_U$, to be the
collection of edges of $G$ having one endpoint in $U$ and the other in the
complement $U^c$.  For each~$e\in E$, define the {\em sign} of $e$ in a cut-set
$c^*_U$ by
\begin{align*}
  \sigma(e,c^*_U):=
  \begin{cases}
    \hfill -1&\text{if $e^{-}\in U$ and $e^{+}\in U^c$,}\\
    \hfill 1&\text{if $e^{-}\in U^c$ and $e^{+}\in U$,}\\
    \hfill 0&\text{otherwise}.
  \end{cases}
\end{align*}
We then write $c^*_U=\sum_{e\in E}\sigma(e,c^*_U)\,e\in\Z E$.  The $\Z$-span of
the cut-sets of $G$ is the {\em (integral) cut space} for $G$, denoted
$\mathcal{C}^*$.  If $U=\{v\}$ for some $v\in V$, then $c^*_v:=c^*_U$ is called
a {\em vertex cut}.  It is well-known that the vertex cuts form a $\Z$-basis for
$\mathcal{C}^*$.

Define the boundary mapping by
\begin{align*}
  \partial:\Z E &\to \Z V_0\\
  e&\mapsto e^{+}-e^{-}.
\end{align*}
for $e\in E$.  We have the following well-known exact sequence (recalling that
we are assuming $G$ is connected):
\[
\xymatrix{
0\to\mathcal{C}\ar[r]&\Z E\ar[r]^{\partial}
&\Z V\ar[r]^(0.4){\mathrm{deg}}&\Z\to0.  
}
\]
A straightforward calculation shows that for each $v\in V$,
\[
\partial(c^*_v)=\Delta(v).
\]
We have the following theorem.
\begin{theorem}[\cite{Bacher}]\label{thm:matroid}
  Let $G$ be an undirected sandpile graph.  Then
  \[
  \sand(G)\approx \Z E/(\mathcal{C}+\mathcal{C^*}).
  \]
\end{theorem}
The following result appears in~\cite{Cori}.
\begin{cor}
  Let $G$ be an undirected planar graph, and let $G^*$ be its dual.  Choosing
  any vertices to serve as sinks, there is an isomorphism of sandpile groups
  \[
  \sand(G)\approx\sand(G^*).
  \]
\end{cor}
\begin{proof}
An orientation
of $G$ induces a dual orientation on $G^*$: if $F$ and $F'$ are adjacent faces
in $G$ (vertices of $G^*$) intersecting along edge $e$, we orient the edge
$e*:=\{F,F'\}$ of $G^*$ as $(F,F')$ if $F$ is to the right of $e$ as one travels
from $e^{-}$ to~$e^{+}$.  Sending $e$ to $e^{*}$ then defines an isomorphism $\Z
E\to\Z E^*$ where $E^*$ denotes the edges of $G^*$.  It is well-known that under
this isomorphism the cycle space (resp., cut space) of $G$ is sent to the
cut space (resp., cycle space) of $G^*$.  The result then follows from
Theorem~\ref{thm:matroid}.  The choice of sink vertices is irrelevant by
Proposition~\ref{prop:sink}.
\end{proof}
\pagebreak 
\begin{remark}\ 
  \begin{enumerate}
    \item The independence of the sandpile group of $G$, up to
      isomorphism, of the choice of sink is also a consequence of
      Theorem~\ref{thm:matroid}.  
    \item Theorem~\ref{thm:matroid} suggests a definition of the sandpile
      group for an arbitrary matroid (\cite{Krueger}).
    \item As noted in \cite{Bacher}, if two undirected (connected) graphs are
      2-isomorphic, then their corresponding matroids are isomorphic.
      (See~\cite{Oxley} for the definition of {\em 2-isomorphism} and a proof of
      the Whitney's 2-isomorphism theorem.)  It then follows from
      Theorem~\ref{thm:matroid} that the sandpile groups for the two graphs
      (having chosen sinks) are isomorphic.
  \end{enumerate}
\end{remark}
\section{Lattice ideals}
Our reference for this section is \cite{Ezra-Bernd}.  Let $A$ be a finitely
generated Abelian group, and let $a_1,\dots,a_n$ be a collection of elements
generating $A$.  Let $Q$ be the subsemigroup of $A$ generated by
$a_1,\dots,a_n$.  In the case where $A$ is finite---the case of special
interest to us---we have that $Q=A$.  Define $\phi\colon\Z^n\to Q$ by $\phi(e_i)=a_i$, and denote its
kernel by $\Lambda$.  Let $\{t_a: a\in Q\}$ be indeterminates, and let
\[
\C[Q]=\myspan_{\C}\{t_a: a\in Q\}
\] 
be the group algebra of $Q$; hence, $t_at_b=t_{a+b}$ for elements
$a,b\in Q$. Letting $R:=\C[x_1,\dots,x_n]$, define a surjection of rings
\ba
\psi\colon R&\to&\C[Q]\\
x_i&\mapsto&t_{a_i}.
\ea
For $c\in\N^n$, we define $x^c=\prod_ix_i^{c_i}$.  Then $\psi(x^c)$ is the group
algebra element $t_b$, where $b=\sum_{i=1}^nc_ia_i$.

For $u\in\Z^n$, we write $u=u^+-u^-$ with $u^+,u^-\in\N^n$
having disjoint support.
\begin{thm}\label{thm:basics}\ 
\be
\item\label{basics:one} The kernel of $\psi$ is the {\em lattice ideal}
\[
I(\Lambda):=\myspan_{\C}\{x^u-x^v: u,v\in\N^n, u-v\in \Lambda\}.
\]
(The vector space span, above, forms an ideal.)  Hence, $\psi$ induces an
isomorphism of $\C$-algebras, $R/I(\Lambda)\approx \C[Q]$. 
\item\label{basics:two} If $\ell_1,\dots,\ell_k$ are generators for the
  $\Z$-module, $\Lambda$,
  then $I(\Lambda)$ is the saturation of 
  \[
  J=\langle
  x^{\ell^+_i}-x^{\ell^-_i}\colon i=1,\dots,k\rangle
  \]
  with respect to
  the ideal generated by the product of the indeterminates, $\prod_{i=1}^nx_i$. Thus,
\[
I(\Lambda)=\{f\in R: (\textstyle\prod_{i=1}^nx_i)^mf\in J \mbox{ for some $m\in\N$}\}.
\]
\item\label{thm:basics3} The Krull dimension of $R/I(\Lambda)$ is
  $n-\dim_{\Z}\Lambda$. 
\ee
\end{thm}
Let $U\subset\N^n$ such that $X:=\{x^u:u\in U\}$ is a $\C$-vector
space basis for $R/I(\Lambda)$. Letting $g:=(a_1,\dots,a_n)\in A^n$,
\[
\psi(X)=\{t_{u\cdot g}: u\in U\}=\{t_a:a\in Q\},
\]
the last equality holding since $R/I(\Lambda)$ and $\C[Q]$ are isomorphic as
vector spaces via~$\psi$.  Now assume that $A$ is a finite group, so that $Q=A$.
Then,~$\psi$ induces a bijection of $X$ with $A$, which endows $X$ with the
structure of a group isomorphic to $A$.  For $u,v\in U$, we define $x^ux^v=x^w$
where $w$ is the unique element of~$U$ for which $w\cdot g=(u+v)\cdot g$.

A choice of a monomial ordering on $R$ gives a natural choice for $U$,
namely, those $u\in\N^n$ such that $x^u$ is not divisible by the
initial term of any element of~$I(\Lambda)$, e.g., not divisible by the initial term
of any element of a Gr\"obner basis for~$I(\Lambda)$.  This will be discussed in
\S\ref{bases}.
\begin{example}
  Let $A=\Z/2\Z\times\Z/3\Z$ with generators $a_1=(1,0)$, $a_2=(0,1)$, and
  $a_3=(1,1)$.  The kernel $\Lambda$ of $\phi\colon\Z^3\to A$ is spanned by
  $(2,0,0)$, $(0,3,0)$, and $(1,1,-1)$.  Hence, the saturation of the ideal
  $(x_1^2-1,x_2^3-1,x_1x_2-x_3)$ gives the lattice ideal $I(\Lambda)$.  Using a
  computer algebra system, one computes
  \[
 I(\Lambda)=(x_1^2-1,x_1x_2-x_3,x_1x_3-x_2,x_2^2-x_3^2,x_2x_3^2-1,x_3^3-x_1).
  \]
 
  By Theorem~\ref{thm:basics}~(\ref{thm:basics3}), one expects a finite set of
  solutions over $\C$ to the equations formed by setting the generators of
  $I(\Lambda)$ equal to zero---there are six.  One vector-space basis for
  $R/I(\Lambda)$ is
  \[
  1,x_1,x_2,x_3, x_2x_3,x_3^2.
  \]
\end{example}
\section{Toppling ideals}\label{toppling ideals}
Let $G$ be a sandpile graph. Identify its vertices
with $\{1,\dots,n+1\}$, where~$n+1$ represents the sink.  To avoid ambiguity, we
will sometimes denote vertex $i$ by $v_i$.  By ordering the
vertices, we thus have the exact sequence for the sandpile group of~$G$,
\[
0\to\Z^n\stackrel{\tD}{\longrightarrow}\Z^n\to\mathcal{S}(G)\to0.
\]
Recall our notation for the reduced Laplacian lattice:
\[
\tLap=\mbox{im}(\tD)=\ker(\Z^n\to\mathcal{S}(G)).
\]
\begin{definition}
  The {\em toppling ideal} for $G$ is the lattice ideal for
  $\tLap$,
\[
I(G):=\myspan_{\C}\{x^u-x^v: u=v\bmod\tLap\}\subset R=\C[x_1,\dots,x_n].
\]
The {\em coordinate ring} for $G$ is $R/I(G)$.
\end{definition}
Thus, by Theorem~\ref{thm:basics}~(\ref{basics:one}), we have the isomorphism of
$\C$-algebras:
\[
R/I(G)\approx\C[\mathcal{S}(G)].
\]

For each nonsink vertex $i$, define the {\em toppling polynomial}
\[
t_i=x_i^{\outdeg(i)-\wt(i,i)}-\textstyle\prod_{j\neq i}x_j^{\wt(i,j)}.
\]
\begin{prop}\label{prop:generators}
  The ideal $I(G)$ is generated by
  the toppling polynomials, $\{t_i\}_{i=1}^n$, and the polynomial
  $x^b-1$ where $b$ is any burning configuration.
\end{prop}
\begin{proof}
  Let $J=(t_i:i=1,\dots,n)+(x^{b}-1)$.  It is clear that $J\subseteq
  I(G)$, and
  by Theorem~\ref{thm:basics}~(\ref{basics:two}), $I(G)$ is the saturation of
  $J$ with respect to the ideal $(x_1\cdots x_n)$.  So it suffices to show that
  $J$ is already saturated with respect to that ideal. Suppose that
  $(x_1\cdots x_n)^kf\in J$ for some $f\in R$ and for some $k$.  For each
  positive integer $m$, consider the monomial $x^{mb}$.  We think of this
  monomial as a configuration of sand with $mb_i$ grains of sand on vertex
  $i$.  If vertex $i$ of this configuration is unstable, we think of firing the vertex
  as replacing $x_i^{mb_i}$ by $x_i^{mb_i-d_i}\textstyle\prod_{j\neq
  i}x_j^{\wt(i,j)}$.  Performing this replacement in $x^{mb}$ gives an
  equivalent monomial modulo $J$.  Recall that every vertex of $G$ is
  connected by a directed path from a vertex in the support of~$b$.  Thus,
  by taking $m$ large enough and firing appropriate vertices, we arrive at a monomial
  $x^{\gamma}$, equivalent to $x^{mb}$ modulo $J$ and corresponding to a
  configuration with at least~$k$ grains of sand at each vertex.  Write
  $x^{\gamma}=x^{\delta}(x_1\cdots x_n)^k$ for some monomial $x^{\delta}$.
  Modulo $J$, we have
  \begin{eqnarray*}
    0 &=& (x_1\cdots x_n)^kf  \\
      &=& x^{\delta}(x_1\cdots x_n)^kf\\
      &=& x^{\gamma}f\\
      &=& x^{mb}f\\
      &=& f.
  \end{eqnarray*}
  Thus, $f\in J$, as required.
\end{proof}
\begin{remark}
  As in the proof of the above theorem, we can identify a monomial $x^{a}$
  with the configuration $a$ on $G$.  If $a\to b$ as
  sandpile configurations, then $x^{a}=x^{b}$ in $R/I(G)$.
\end{remark}
\begin{remark}
  The toppling ideal was introduced by Cori, Rossin, and Salvy \cite{CRS}.  They
  considered only undirected graphs and defined the ideal via generators.  For
  an undirected graph, the all-$1$s vector is a burning script, so
  Proposition~\ref{prop:generators} shows that our definition coincides with theirs
  in the case of an undirected graph.
\end{remark}
\begin{figure}[ht] 
\begin{tikzpicture}[scale=1.3]
\SetVertexMath
\GraphInit[vstyle=Art]
\SetUpVertex[MinSize=3pt]
\SetVertexLabel
\tikzset{VertexStyle/.style = {%
shape = circle,
shading = ball,
ball color = black,
inner sep = 1.5pt
}}
\SetUpEdge[color=black]
\Vertex[LabelOut,Lpos=180, Ldist=.1cm,x=0,y=1]{v_1}
\Vertex[LabelOut,Lpos=0, Ldist=.1cm,x=1,y=1]{v_2}
\Vertex[LabelOut,Lpos=180, Ldist=.2cm,x=0,y=0]{v_3}
\Vertex[LabelOut,Lpos=0, Ldist=.2cm,x=1,y=0]{v_4}
\Edge[](v_1)(v_2)
\Edge[style={-triangle 45}](v_1)(v_3)
\Edge[style={-triangle 45}](v_2)(v_4)
\Edge[style={-triangle 45}](v_3)(v_4)
\Edge[style={-triangle 45}](v_3)(v_2)
\fill[color=white] (0.5,0.5) circle (0.1cm);
\draw (0.5,0.5) node{{\small 2}};
\end{tikzpicture}
\caption{Sandpile graph $G$ with sink $v_4$.}\label{fig:ex1}
\end{figure}
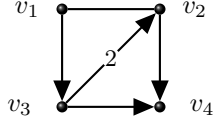
\begin{example}\label{ex:ideals}
The sandpile graph $G$ in Figure \ref{fig:ex1} has a burning script
$\sigma=(1,2,1)$ and corresponding burning configuration $b=(0,1,2)$. Thus,
\[
I(G) = (x_1^2 - x_2 x_3, x_2^2 - x_1, x_3^3 - x_2^2, x_2 x_3^2 - 1).
\]
\end{example}
\begin{definition}
  Let $f\in R=\C[x_1,\dots,n]$, and let $x_{n+1}$ be another
  indeterminate.  The {\em homogenization} of $f$ with respect to $x_{n+1}$ is the
   homogeneous polynomial
   \[
   f^h:=x_{n+1}^{\deg f}f\left(\frac{x_1}{x_{n+1}},\dots,\frac{x_n}{x_{n+1}}\right).
   \]
   If $I\subseteq R$ is an ideal, the {\em homogenization} of $I$ with respect
   to $x_{n+1}$ is the ideal
   \[
   I^h:=(f^h:f\in I).
   \]
\end{definition}

Now consider the exact sequence corresponding to the full Laplacian,
\[
\Z^{n+1}\stackrel{\Delta}{\longrightarrow}\Z^{n+1}\to\Z^{n+1}/\Lap\to0
\]
recalling the notation for the Laplacian lattice, $\Lap:=\im(\Delta)$.
Let $S=\C[x_1,\dots,x_{n+1}]$ and consider the lattice ideal for
$\Lap$.  We here introduce the homogeneous version of the toppling ideal.
\begin{definition}  The {\em homogeneous toppling ideal} for $G$ is
\[
I_h(G):=\myspan_{\C}\{x^u-x^v:u=v\bmod\Lap\}\subset
S=\C[x_1,\dots,x_{n+1}].
\]
The {\em homogeneous coordinate ring} for $G$ is $S/I_h(G)$.
\end{definition}
The following proposition is straightforward.  Its hypothesis is satisfied for
any Eulerian graph and, in particular, for any undirected graph.  Moreover, given
any sandpile graph with sink $s$, removing all out-edges from $s$ creates a new
sandpile graph with the same sandpile group and for which the hypothesis of the
proposition holds.
\begin{prop}\label{prop:homog}
  If $\Delta(v_{n+1})\in\myspan_{\Z}\{\Delta(v_1),\dots,\Delta(v_n)\}$,
  then $I_h(G)=I(G)^h$.
\end{prop}
\begin{example}\label{example:bad sink}
  The graph $G$ in Figure~\ref{fig:homog} does not satisfy the hypothesis of
  Proposition~\ref{prop:homog}.  Regarded as a sandpile graph with sink $v_1$,
  the toppling ideal for $G$ is $(x_1^2-1)$.  As a sandpile graph with sink
  $v_2$, its toppling ideal is $(x_2^3-1)$.  Its homogeneous toppling ideal is
  $I_h(G)=(x_1-x_2)$, equivalent to that of the undirected graph with a single
  edge connecting $v_1$ and $v_2$ (or equivalent to that of the directed graph consisting of
  a single directed edge connecting $v_1$ to $v_2$).
\end{example}
\begin{figure}[ht] 
\begin{tikzpicture}
\SetVertexMath
\GraphInit[vstyle=Art]
\SetUpVertex[MinSize=3pt]
\SetVertexLabel
\tikzset{VertexStyle/.style = {%
shape = circle,
shading = ball,
ball color = black,
inner sep = 1.3pt
}}
\SetUpEdge[color=black]
\Vertex[LabelOut,Lpos=180,
Ldist=.1cm,x=0,y=0]{v_1}
\Vertex[LabelOut,Lpos=0,
Ldist=.1cm,x=2,y=0]{v_2}
\Edge[style={-triangle 45, bend right=30}](v_1)(v_2)
\Edge[style={-triangle 45, bend right=30}](v_2)(v_1)
\fill[color=white] (1,0.30) circle (0.12cm);
\draw (1,0.30) node{{\small 3}};
\fill[color=white] (1,-0.30) circle (0.12cm);
\draw (1,-0.30) node{{\small 2}};
\end{tikzpicture}
\caption{Graph $G$.}\label{fig:homog}
\end{figure}
\begin{remark}
  In general, homogenizing the generators of an ideal does not produce a
  complete set of generators for the homogenized ideal.  For instance, the graph
  in Example~\ref{example:genus2} has toppling ideal generated by $4$
  polynomials, whereas its homogeneous toppling ideal is minimally generated by
  $6$ polynomials.
\end{remark}

\begin{thm}\label{thm:sandpile lattice ideals}
  Let $\tLap$ be any submodule of $\Z^n$ having rank $n$.  Then
  there exists a sandpile graph whose reduced Laplacian
  lattice is $\tLap$.  Every lattice ideal defining a finite set of points is
  the lattice ideal associated with the reduced Laplacian of some sandpile
  graph.
  \end{thm}
\begin{proof}
  In light of Theorem~\ref{thm:basics}~(\ref{thm:basics3}), it suffices to prove
  that given an $n\times n$ matrix~$M$ of rank $n$, there exists a matrix $M'$
  with the same integer column span as $M$ and which is the reduced Laplacian
  matrix of some sandpile graph.  Recall that a matrix~$M'$ is the reduced
  Laplacian of a directed multigraph if and only if (i) $\deg(c) \ge 0$ for each
  column~$c$ of $M'$, (ii) $M_{ii}' > 0$, (iii) $M_{ij}' \le 0$ for $i \ne j$.
  (If $c$ is a column vector of a matrix, then $\deg(c)$ is the sum of the
  entries of $c$.) If in addition $M'$ has full rank, then its corresponding
  graph has a globally accessible vertex by the Matrix-Tree Theorem. The desired
  matrix $M'$ is produced by Algorithm~\ref{algo:get-lap}, stated below.  It
  proceeds in three steps, modifying the columns of $M$ using only invertible
  integral column operations.

  First, since $M$ has rank $n$, not all columns have $\deg(c) = 0$. Using the
  Euclidean algorithm, by adding multiples of one column to another, we set $\deg(c)$
  to $0$ for all but one column $c$ of $M$ (line~1). By possibly moving and
  negating that column, we have that $\deg(c_i) = 0$ for all but the first
  column $c_1$, for which $\deg(c_1) > 0$.

  Next, we repeat the Euclidean algorithm another $(n-2)$ times, now on the
  super-diagonal entries of each of the first $(n-2)$ rows in turn (lines
  2--9). Again by adding multiples of one column to another, we have every
  entry more than one row above the diagonal set to $0$. Note that since this
  step only involves addition of columns whose degree is already zero, the
  column degrees are not affected. Additionally, since $M$ had rank $n$ and the
  last $(n-1)$ columns have degree zero, we have that each of these columns has
  a nonzero superdiagonal entry. Now by negating columns where necessary, we
  may assume that the nonzero superdiagonal entry of each column is negative.

  At this point, the last column satisfies (i)--(iii). Assuming the last $r$
  columns $c_{n-r+1}, \ldots, c_n$ satisfy (i)--(iii) for $r \le n-2$, we claim
  that for any $1 \le s \le r$ there is a vector $v^s \in
  \myspan_{\Z}\{c_{n-r+1},\ldots, c_n\}$ with $v^s_n < 0$ and $v^s_{n-s} > 0$
  and with all other entries zero. For $r = 1$, the vector $v^1$ is obtained by
  negating $c_n$, so we proceed by induction on $r$. With the hypotheses
  satisfied for some $r$, we already have appropriate vectors $v^1,\ldots,
  v^{r-1}$. To obtain $v^r$, note that $-c_{n-r+1}$ has a positive entry in row
  $(n-r)$, so by adding appropriate multiples of the $v^s$ for $s<r$, we produce
  the desired column vector.
  
  Given that such vectors $v^s$ exist, it is clear that we may iteratively
  correct the columns from right to left by adding multiples of the higher
  indexed columns. We now give this algorithm explicitly.  In what follows,
  $v[j]$ denotes the $j$-th entry of the column vector $v$, and the Euclidean
  algorithm terminates when run in-place on some set of integers, $S$, once a
  single element of $S$ equals the positive GCD of the elements of $S$ and
  every other element of $S$ is zero.
\medskip

\begin{myalgorithm}\ \label{algo:get-lap}
  \smallskip
\REQUIRE An $n\times n$ matrix $M$ of rank $n$ with columns $c_1,\dots, c_n$.
\ENSURE The reduced Laplacian matrix $\tD(G)$ of a directed
multigraph $G$ such that $\tD(G) = MU$ for some invertible
integral matrix $U$.

\STATE Run the Euclidean algorithm on the set $S = \{\deg(c_k)\}$ by subtracting 
one column from another at each step. \label{getlap:columnsum} Swap
columns so that $\deg(c_1) = \gcd(S)$ and $\deg(c_i) = 0$ for $i > 1$.
\label{getlap:c1}
\FOR{$k \leftarrow 2$ \UPTO $n-1$}
    \STATE Run the Euclidean algorithm on the set $S = \{c_i[k-1]: i \ge k\}$ by
    subtracting one column from another at each step.

    \STATE Swap columns so that $c_k[k-1] = \gcd(S)$ and $c_i[k-1] = 0$ for $i > k$
    \STATE $c_k \leftarrow -c_k$
\ENDFOR \label{getlap:zeros}
\IF{$c_n[n-1] > 0$}
    \STATE $c_n \leftarrow -c_n$
\ENDIF \label{getlap:negate}
\FOR{$k \leftarrow n-1$ \DOWNTO $1$} \label{getlap:loop}
    \FOR[this loop is not entered until $k \le n-2$]{$i \leftarrow k+2$ \UPTO $n$}
        \WHILE{$c_k[i-1] > 0$}
            \STATE $c_k \leftarrow c_k + c_i$
        \ENDWHILE
    \ENDFOR \label{getlap:k-n}
    \STATE $v \leftarrow -c_{k+1}$
    \FOR[this loop is not entered until $k \le n-2$]{$i \leftarrow k+2$ \UPTO $n$}
        \STATE $v \leftarrow |c_i[i-1]|\cdot v + v[i-1]\cdot c_i$
    \ENDFOR \label{getlap:v}
    \WHILE{$c_k[k] \le 0$ or $c_k[n] > 0$}
        \STATE $c_k \leftarrow c_k + v$
    \ENDWHILE
    \label{getlap:k+n}
\ENDFOR
\RETURN $[c_1 \cdots c_n]$
\end{myalgorithm}
\end{proof}

For the sake of the following corollary, a {\em weighted path graph} 
$P=u_1\dots u_k$ is a graph with vertex set $\{u_1,\dots,u_k\}$ and weighted
edges $\{(u_i,u_{i+1}): 1\leq i<k\}$.  If $F$ and $F'$ are weighted digraphs,
their {\em graph sum} is the graph $F+F'$ whose weighted adjacency matrix is the
sum of those for $F$ and $F'$.
\begin{cor}\label{alg-cor}
Let $G$ be a sandpile graph with vertex set $V = \{v_1, \ldots, v_{n+1} \}$ and
sink $v_{n+1}$.  Then there exists a weighted path graph $P = v_{n}v_{n-1}\cdots
v_1v_{n+1}$ and a directed acyclic graph $D$ on the nonsink vertices $\tV$
oriented from lower-indexed vertices to higher such that the graph sum $G' = P +
D$ has the same Laplacian lattice as $G$.
\end{cor}
The above simply states the form of the graph given by the output of
Algorithm~\ref{algo:get-lap}. The graph $G'$ of Corollary~\ref{alg-cor} is not
uniquely determined.  For instance, by iterating line~21 of
Algorithm~\ref{algo:get-lap} more times than necessary, one may generate
infinitely many graphs $G'$ of the form described in the corollary, each with
Laplacian lattice $\Lap$.  
\begin{example}\label{example:same laplacian}
  One sandpile graph of the form given by Corollary~\ref{alg-cor} with the same
  Laplacian lattice as the sandpile graph $G$ from Example~\ref{ex:ideals}
  is $G'$ appearing in Figure~\ref{fig:algout}.
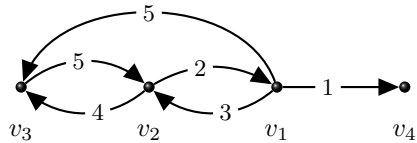
\begin{figure}[ht] 
\begin{tikzpicture}[scale=1.7]
\SetVertexMath
\GraphInit[vstyle=Art]
\SetUpVertex[MinSize=3pt]
\SetVertexLabel
\tikzset{VertexStyle/.style = {%
shape = circle,
shading = ball,
ball color = black,
inner sep = 1.5pt
}}
\SetUpEdge[color=black]
\Vertex[LabelOut,Lpos=270, Ldist=.3cm,x=0,y=0]{v_3}
\Vertex[LabelOut,Lpos=270, Ldist=.3cm,x=1,y=0]{v_2}
\Vertex[LabelOut,Lpos=270, Ldist=.3cm,x=2,y=0]{v_1}
\Vertex[LabelOut,Lpos=270, Ldist=.3cm,x=3,y=0]{v_4}
\Edge[style={-triangle 45, bend left=40}](v_3)(v_2)
\Edge[style={-triangle 45, bend left=40}](v_2)(v_3)
\fill[color=white] (0.45,0.2) circle (0.1cm);
\draw (0.45,0.2) node{{\small 5}};
\fill[color=white] (0.6,-0.2) circle (0.1cm);
\draw (0.6,-0.2) node{{\small 4}};
\Edge[style={-triangle 45, bend left=30}](v_2)(v_1)
\Edge[style={-triangle 45, bend left=40}](v_1)(v_2)
\fill[color=white] (1.4,0.15) circle (0.1cm);
\draw (1.4,0.15) node{{\small 2}};
\fill[color=white] (1.6,-0.2) circle (0.1cm);
\draw (1.6,-0.2) node{{\small 3}};
\Edge[style={-triangle 45}](v_1)(v_4)
\fill[color=white] (2.4,0.0) circle (0.1cm);
\draw (2.4,0.0) node{{\small 1}};
\Edge[style={-triangle 45, bend right=70}](v_1)(v_3)
\fill[color=white] (1,0.58) circle (0.1cm);
\draw (1,0.58) node{{\small 5}};
\end{tikzpicture}
\caption{The sandpile graph $G'$ for Example~\ref{example:same laplacian}.}
\label{fig:algout}
\end{figure}
\end{example}
\begin{question}
When is it the case that a submodule of $\Z^n$ with rank $n$ is the reduced
Laplacian lattice of an {\em undirected} graph?  It is not always the case.  For
instance, Figure~\ref{fig:gor} is a directed sandpile graph whose lattice ideal
is Gorenstein (cf.~\S\ref{section:gor}) and with sandpile group of order $5$.
By Theorem~\ref{thm:loopy}, any undirected graph with Gorenstein lattice ideal must be a tree and
would thus have sandpile group of order $1$.
\end{question}
\section{Gr\"obner bases of toppling ideals}\label{bases}
We recommend~\cite{CLO} as a general reference for the theory of Gr\"obner bases
needed in this section.
Let
$R=\C[x_1,\dots,x_n]$.   
\begin{definition}
  A {\em monomial order}, $>$, on $R$ is a total ordering on the monomials of $R$
  satisfying
  \begin{enumerate}
    \item If $x^a>x^b$, then $x^{c+a}>x^{c+b}$ for all $c\geq0$;
    \item $1=x^0$ is the smallest monomial.
  \end{enumerate}
\end{definition}
\begin{example}
  The following are the most common examples of monomial orders:
  \begin{enumerate}
    \item Lexicographic ordering, \verb+lex+, is defined by $x^a>x^b$ if the
      left-most nonzero entry of $a-b$ is positive (i.e., more of the earlier
      indeterminates).
    \item Degree lexicographic ordering, \verb+deglex+, is defined by $x^a>x>b$
      if $\deg(a)>\deg(b)$ or if $\deg(a)=\deg(b)$ and the left-most nonzero entry of $a-b$ is
      positive (i.e., order by degree and break ties with \verb+lex+).
    \item Degree reverse lexicographic ordering, \verb+grevlex+, is defined by
      $x^a>x^b$ if $\deg(a)>\deg(b)$ or if $\deg(a)=\deg(b)$ and the right-most
      nonzero entry of $a-b$ is negative (i.e., order by degree then break ties by checking which
      monomial has fewer of the later indeterminates).
  \end{enumerate}
\end{example}
A monomial multiplied by a constant is called a {\em term}.  Once a monomial
ordering is fixed, write $\alpha\,x^a>\beta\,x^b$ for two terms if $\alpha$ and
$\beta$ are nonzero and $x^a>x^b$. Each $f\in R$ is a sum of terms corresponding
to distinct monomials.  We denote the leading term---the largest term with
respect to the chosen monomial ordering---by~$\lt(f)$.  
  
  \begin{definition}
     Fix a monomial ordering on $R$ and let $f,g\in R$.  The {\em
     $S$-polynomial} for the pair $(f,g)$ is 
     \[ S(f,g) = \frac{\mathrm{lcm}(\lt(f),\lt(g))}{\lt(f)}\,f-
     \frac{\mathrm{lcm}(\lt(f),\lt(g))}{\lt(g)}\,g.  
     \]
    \end{definition}

\begin{definition}
    Fix a monomial ordering on $R$, and let $I$ be an ideal of $R$.  A finite
    subset $\Gamma$ of $I$ is a {\em Gr\"obner basis} for $I$ with respect to the
    given monomial ordering if for all $f\in I$ there is a $g\in\Gamma$ such
    that $\lt(g)$ divides $\lt(f)$.
\end{definition}
Let $\Gamma=\{g_1,\dots,g_m\}$ be the Gr\"obner basis for an ideal $I\subseteq
R$ with respect to some monomial ordering, and let $f\in\R$.  If $f$ has a term
$m$ divisible by $\lt(g_i)$ for some $i$, then replace $f$ by
$f-\frac{m}{\lt(g_i)}\,g_i$.  A standard result in the theory of Gr\"obner
bases is that by repeating this process one arrives at a remainder $r$ that is
unique with respect to the property that (i) $r=f+g$ for some $g\in I$ and (ii)
$r$ has no terms divisible by any leading term of an element of $\Gamma$.  We
call this remainder the {\em reduction} or {\em normal form} of $f$ with respect
to the Gr\"obner basis $\Gamma$.
\begin{notation}
  The reduction of $f$ with respect to $\Gamma$ is denoted by $f\,\%\,\Gamma$.  If $g\in
  R$, we write
  $f\,\%\,g$ for the special case in which $I=(g)$ and $\Gamma=\{g\}$.
\end{notation}

\begin{prop}
  Fix a monomial ordering on $R$, and let $I$ be an ideal of $R$.   The
  following are equivalent for a finite subset $\Gamma$ of $I$: 
  \begin{enumerate}
    \item $\Gamma$ is a Gr\"obner basis with respect to the given ordering;
    \item there is an equality of ideals: $(\lt(g):g\in \Gamma)=(\lt(f):f\in
      I)$;
    \item each $f\in I$ may be reduced to $0$ by $\Gamma$, i.e.,
      $f\,\%\,\Gamma=0$;
    \item for all $g,g'\in \Gamma$, the $S$-polynomial $S(g,g')$ reduces to $0$
      by $\Gamma$ and $\Gamma$ is a generating set for $I$.
  \end{enumerate}
\end{prop}
The last criterion is essentially Buchberger's algorithm for calculating a
Gr\"obner basis: start with any generating set for $I$, and if
$f:=S(g,g')\,\%\,\Gamma\neq0$ for some pair of generators $g$ and $g'$, add $f$
to the set of generators and check the $S$-pairs again.  The process eventually
stops.
\begin{definition}
  Fix a monomial ordering on $R$ and let $I$ be an ideal of $R$.  The set of
  monomials of $R$ that are not divisible by the leading term of a Gr\"obner
  basis element for $I$ with respect to the given ordering is called the {\em
  normal basis} for $R/I$.
\end{definition}
\noindent By Macaulay's theorem (Theorem 15.3, \cite{Eisenbud}), a normal basis is a vector space basis for $R/I$.

We now introduce an appropriate monomial ordering for sandpiles, due to
Cori, Rossin, and Salvy, \cite{CRS}.
\begin{definition}\label{def:ordering} 
  Let $G$ be a sandpile graph with vertices $\{v_1,\dots,v_{n+1}\}$ and with
  sink $v_{n+1}$.  A {\em sandpile monomial ordering} on $R=\C[x_1,\dots,x_n]$
  is any \verb+grevlex+ ordering for which $x_i>x_j$ if the length of the
  shortest path from vertex~$v_j$ to the sink is no greater than that for $v_i$.
  Given a sandpile monomial ordering $>$ on $R$, the sandpile monomial ordering
  on $S=\C[x_1,\dots,x_{n+1}]$ {\em compatible} with $>$ is the \verb+grevlex+
  order extending $>$ for which $x_i>x_{n+1}$ for $i=1,\dots,n$.
\end{definition}

\begin{prop}
  With notation as in Definition~\ref{def:ordering}, let $>$ be a sandpile
  monomial ordering on $R$, extended to a compatible sandpile monomial ordering
  on $S$.  Let $I\subset R$ be the toppling ideal for $G$.
  \begin{enumerate}
    \item Let $\Gamma$ a Gr\"obner basis for $I$ with respect to $>$,
      and let $\Gamma^h$ be the subset of~$S$ formed by homogenizing each
      element of $\Gamma$.  Then $\Gamma^h$ is a Gr\"obner basis for the
      homogenization $I^h\subset S$.
    \item The normal bases for $R/I$ and for $S/(I^h+(x_{n+1}))$
      consist of the same set of monomials.  Hence, $R/I$ and $S/(I^h+(x_{n+1}))$
      are isomorphic as vector spaces.
  \end{enumerate}
\end{prop}
\begin{proof}
  The first part of the proposition is a general result for \verb+grevlex+
  orderings (cf.~Exercise 5, \S8.4, \cite{CLO}).  It is straightforward to
  check that if $f\in R$, then $\lt(f)=\lt(f^h)$, from which the second part
  follows.
\end{proof}

\noindent
\fbox{
\parbox{\textwidth}{
{\sc assumption:} For the rest of \S5, we fix a sandpile graph $G$ as in
Definition~\ref{def:ordering}, and a sandpile monomial ordering on~$R$.  We
assume the vertices are numbered so that $x_i>x_j$ if
$i<j$.} 
}
\medskip

The utility of a sandpile monomial ordering becomes apparent when one considers
topplings of sandpiles.  

\begin{prop} 
Let $a,b\in\N\tV$ be distinct configurations on $G$ such that $a\to
b$, i.e.,~$b$ is obtained from $a$ by a sequence of vertex firings.  Then,
$x^a>x^b$ with respect to the sandpile monomial ordering we have fixed on
$R$.
\end{prop}
\begin{proof}
Each vertex firing deceases the size of the corresponding monomial.  The
reason is that either the vertex firing shoots sand into the sink,
decreasing the total degree of the corresponding monomial, or it shoots sand
to a vertex closer to the sink, in which case the corresponding monomial
has more of the later indeterminates. 
\end{proof}

We now proceed to compute a Gr\"obner basis for the toppling ideal.  Let
\begin{eqnarray*}
  E\colon\Z\tV&\to&R\\
  \ell&\mapsto& x^{\ell^{+}}-x^{\ell^{-}}.
\end{eqnarray*}
Then define $\mathcal{T}=E\circ\tD:\Z\tV\to R$.  Thus,
$\mathcal{T}(v_i)$ is the $i$-th toppling polynomial, defined earlier, and for
any configuration $c$, we have $x^c\,\%\,\mathcal{T}(v_i)=x^{c'}$ where~$c'$ is
the configuration obtained from $c$ by firing $v_i$ until $v_i$ is stable.
Morever, if $\sigma$ is a firing-script, then $x^c\,\%\,\mathcal{T}(\sigma)$
yields the monomial corresponding to the configuration formed by firing $\sigma$
as many times as legal from $c$.  The following theorem appears
in the Bachelor's thesis of the second author,~\cite{Perlman}.

\begin{thm}\label{thm: groebner basis}
Let $b$ be a burning configuration, and let $\sigma_b$ be its script.  Then 
\[
\Gamma_b = \{\mathcal{T}(\sigma):0\leq\sigma\leq\sigma_b\}
\]
is a Gr\"obner basis for $I(G)$.
\end{thm}
\begin{proof}
We have $\im(\mathcal{T})\subset I(G)$ by definition of
$I(G)$.  On the other hand, $\mathcal{T}(v_i)$ is the $i$-th toppling
polynomial and $\mathcal{T}(\sigma_b)=x^b-1$.  So
$I(G)=\myspan_{\C}\{\im(\mathcal{T})\}$ by
Proposition~\ref{prop:generators}.  
 
We need to show that all $S$-polynomials of $\Gamma_b$ reduce to $0$ by
$\Gamma_b$.  Let $\sigma_1$ and $\sigma_2$ be scripts with
$\sigma_1,\sigma_2\leq\sigma_b$.  Write
\[
\mathcal{T}(\sigma_i)=x^{c_i^{+}}-x^{c_i^{-}}
\]
for $i=1,2$ where $c_i^{-}$ is the configuration obtained from $c_i^{+}$ by
firing script $\sigma_i$.  Hence,~$x^{c_i^{+}}$ is the leading term of
$\mathcal{T}(\sigma_i)$ for each $i$.
  Define
  \[
  x^{a_i}=\frac{\mathrm{lcm}(x^{c_1^{+}},x^{c_2^{+}})}{x^{c_i^{+}}}
  \]
  for $i=1,2$ so that $a_1+c_1^{+}=a_2+c_2^{+}=c$ for some configuration $c$.
  We must show that the $S$-polynomial,
  \begin{eqnarray*}
    S(\mathcal{T}(\sigma_1),\mathcal{T}(\sigma_2))&=& x^{a_1}\mathcal{T}(\sigma_1)-
    x^{a_2}\mathcal{T}(\sigma_2)\\
    &=& x^{a_2+c_2^{-}}-x^{a_1+c_1^{-}},
  \end{eqnarray*}
  reduces to $0$.
  Since both scripts $\sigma_1$ and $\sigma_2$ are legal from $c$, so is
  the script $\sigma=\max(\sigma_1,\sigma_2)$ defined by
  $\sigma_v=\max(\sigma_{1,v},\sigma_{2,v})$.  Note that
  $\sigma\leq\sigma_b$.  Letting $c'$ be the configuration
  obtained by firing $\sigma$, we have the sequence of legal script-firings
  \[
  a_i+c_i^{+}\stackrel{\sigma_i}{\longrightarrow}a_i+c_i^{-}
  \stackrel{\sigma-\sigma_i}{\longrightarrow}c'
  \]
  for $i=1,2$, which shows that the $S$-polynomial reduces to $0$ using the
  elements $\mathcal{T}(\sigma-\sigma_i)$ for $i=1,2$.
\end{proof}
\begin{remark}
  In the case of an undirected graph, one may take the burning script to be
  $\vec{1}$, the vector whose components are all ones.  Thus, the script-firings
  that are relevant in constructing the Gr\"obner basis, described in the
  statement of the previous theorem, can be identified with firing subsets of
  vertices (none more than once).  The paper \cite{CRS} goes further, in this
  case, to describe a {\em minimal} Gr\"obner basis, i.e., one in which each
  member has the property that none of its terms is divisible by the leading
  term of any other member.  It consists of the subset of the Gr\"obner basis
  elements described in the previous theorem corresponding to $X\subseteq\tV$
  such that the subgraphs of $G$ induced by $X$ and by $\tV\setminus X$ are each
  connected.  It would be interesting to see if this result could be generalized
  to the case of directed graphs.
\end{remark}
\begin{thm}\label{thm:normal basis}
  Each nonnegative configuration is equivalent to a unique superstable sandpile
  modulo
  $\tLap$, and 
  \[
  \{x^c:c\mbox{ is a superstable configuration}\}
  \]
  is the normal basis for $R/I(G)$ with respect to the sandpile monomial
  ordering.
\end{thm}
\begin{proof}
  Two nonnegative configurations are equivalent modulo $\tLap$ if and only if
  their corresponding monomials are equivalent modulo the toppling ideal,
  $I(G)$.  In detail, first let $c_1,c_2\in\N^n$ and suppose 
  \[
  c_1-c_2=\ell=\ell^{+}-\ell^{-}\in\tLap.
  \]
  Then $c_1\geq\ell^{+}$ and $c_2\geq\ell^{-}$.  Define
  $e=c_1-\ell^{+}=c_2-\ell^{-}\geq0$. Then
  \[
  x^{c_1}-x^{c_2}=x^e(x^{\ell^{+}}-x^{\ell^{-}})\in
  I(G).
  \]
  Conversely, suppose $x^{c_1}-x^{c_2}\in I(G)$. Identify the sandpile
  group $\mathcal{S}(G)$ with $\Z^n/\tLap$. Let
  \begin{eqnarray*}
    \psi\colon\C[x_1,\dots,x_n]&\to&\C[\Z^n/\tLap]\\
    x_i&\mapsto&t_{e_i}
  \end{eqnarray*}
  be the mapping into the group algebra where $e_i$ is
  the image of the $i$-th standard basis vector for $\Z^n$.  Then
  $I(G)=\ker\psi$.  Hence,
  \[
  0=\psi(x^{c_1}-x^{c_2})= t_{c_1}-t_{c_2}.
  \]
  In other words, $c_1-c_2\in\tLap$.

  Now let $c$ be any nonnegative configuration.  Since
  $x^c\,\%\,\mathcal{T}(\sigma)=x^{c'}$ where $c'$ is obtained by firing the script
  $\sigma$ as many times as is legal, the normal form for $x^c$ with respect to
  the sandpile monomial ordering is superstable.  Since the normal form is
  unique, so is this superstable element.
  \end{proof}
\begin{remark}\label{super}
As noted in \S\ref{toppling ideals}, we have
$R/I(G)\approx\C[\mathcal{S}(G)]$.  Hence, by the previous theorem,
we see that the sandpile group can be thought of as the set of superstables
where the sum of superstables $c_1$ and $c_2$ is taken to be
$\log(x^{c_1}x^{c_2}\,\%\,I(G))$.
\end{remark}

The following can be found in \cite{Holroyd} for the case of Eulerian graphs.
Here we extend the result to general sandpile graphs (for which the underlying
graph is a directed multigraph).
\begin{cor}\label{cor:duality}
A configuration $c$ is superstable if and only if $\cmax-c$ is recurrent.
\end{cor}
\begin{proof}
  By Theorems~\ref{thm:normal basis} and~\ref{thm:main isomorphism}, the number
  of superstable configurations is equal to the number of recurrent
  configurations.  Thus, is suffices to show that if $c$ is superstable, then
  $\cmax-c$ is recurrent.

  Let $b$ be a burning configuration for $G$ with burning script $\sigma_b$.
  Since $c$ is superstable, there exists $u_1\in\supp(\sigma_b)$ such that
  $(c-\tD\,\sigma_b)_{u_1}<0$.  Similarly, there exists
  $u_2\in\supp(\sigma_b-u_1)$ such that $(c-\tD(\sigma_b-u_1))_{u_2}<0$.
  Continuing, we find a sequence of nonsink vertices $u_1,\dots,u_k$ such that
  $\sum_{i=1}^{k}u_i=\sigma_b$ and for $1\leq j\leq k$,
  \[
   \left(c-\tD(\sigma_b-\textstyle\sum_{i=1}^{j-1}u_i)\right)_{u_j}<0.
  \]
  It follows that $u_1,\dots,u_k$ is a legal firing sequence for $\cmax-c+b$,
  reducing $\cmax-c+b$ to $\cmax-c$.
  Hence, $\cmax-c$ is recurrent by Theorem~\ref{thm:speer}.
\end{proof}
In light of Corollary~\ref{cor:duality}, we say that the superstables are {\em
dual} to the recurrents.

\section{Zeros of the toppling ideal}
Given any ideal $I\in R=\C[x_1,\dots,x_n]$, the {\em set of zeros} of $I$ is\
\[
Z(I)=\{p\in\C^n: f(p)=0\mbox{ for all $f\in I$}\}.
\]

In this section, our goal is to describe the set of zeros of the toppling ideal.
\begin{prop}\label{thm:zeros}
  Let $G$ be a sandpile graph.  Then the set of zeros of its toppling ideal,
  $I(G)$, is finite.
\end{prop}
\begin{proof}
  Since $I(G)$ is the lattice ideal for a square matrix of full rank,
  Theorem~\ref{thm:basics}~(\ref{thm:basics3}) guarantees that the set of zeros is finite.  However,
  we will give a direct proof.
  We have seen that
  \[
  R/I(G)\approx \C[\mathcal{S}(G)],
  \]
  and thus, $R/I(G)$ is a finite-dimensional vector space over $\C$.  For
  each indeterminate $x_i\in R$, consider the powers $1,x_i,x_i^2,\dots$  By
  finite-dimensionality, the image of these powers in the quotient ring are
  linearly dependent.  This means there is a polynomial $f_i$ in one variable
  such that $f_i(x_i)\in I(G)$.  Each $f_i$ will have a finite number of zeros,
  and thus, for each $i$, we see that the there are a finite number of possible
  $i$-th coordinates for any zero of the toppling ideal.
\end{proof}

\begin{remark} In fact, the $i$-th coordinates of the zeros of the toppling ideal are the
eigenvalues of the multiplication mapping
\begin{eqnarray*}
  R/I(G)&\to& R/I(G)\\
  g&\mapsto&x_ig
\end{eqnarray*}
See~\cite{Cox}, for instance.
\end{remark}

\subsection{Orbits of representations of Abelian groups}
\subsubsection{Affine case.}\label{affine case}
Let $\{a_1,\dots,a_n\}$ be generators (not necessarily distinct) for a finite
Abelian group, $A$.  Consider the exact sequence
\begin{eqnarray}\label{exact}
0\to \Lambda\to\Z^n&\to&A\to0\\
e_i&\mapsto&a_i\nonumber
\end{eqnarray}
where $\Lambda$ is defined as the kernel of the given mapping $\Z^n\to A$.
Taking duals by applying $\mbox{Hom}_{\Z}(\ \cdot\ ,\C^{\times})$ gives the
sequence
\begin{equation}\label{exact-dual}
1\leftarrow \Lambda^*\leftarrow (\C^{\times})^n\leftarrow A^*\leftarrow 1,
\end{equation}
where $A^*$ is the character group of $A$.
\medskip

\begin{remark}~ 
  \begin{enumerate}
    \item Exactness of (\ref{exact-dual}) is not immediate.  The exactness
      at $\Lambda^*\leftarrow(\C^{\times})^n$ follows because $\C^{\times}$ is a
      {\em divisible} Abelian group.  An Abelian group $B$ is divisible if for
      all $a\in B$ and positive integers $n$ there exists $b\in B$ such that
      $nb=a$.  (For the multiplicative group $\C^{\times}$, each element has
      an $n$-th root.)  Applying $\mbox{Hom}_{\Z}(\ \cdot\ ,B)$ to an exact sequence
      of Abelian groups ($\Z$-modules) always gives an exact sequence precisely
      when $B$ is divisible.  The proof of this, in general, is not immediate.
      However, in the case in which we are most concerned, the exactness is easy
      to establish.  Suppose $A=\mathcal{S}(G)$ is the sandpile group of a
      sandpile graph, and suppose $\Lambda$ is the reduced Laplacian
      lattice, $\tLap=\im(\tD)\hookrightarrow\Z^n$.
      We would like to show that the natural map, given by composition, 
      \[
      \mbox{Hom}(\Z^n,\C^{\times})\to\mbox{Hom}(\tLap,\C^{\times})
      \]
      is surjective.  Let $\phi\colon \tLap\to\C^{\times}$ be given.  Since the
      reduced Laplacian has full rank, given $v\in\Z^n$, there exist unique
      rational numbers $\alpha_{\ell}$ such that
      $v=\sum_{\ell}\alpha_{\ell}\ell$, with the sum going over a basis for
      $\tLap$
      (say, over the columns of the reduced Laplacian).  Then define
      $\tilde{\phi}\colon \Z^n\to \C^{\times}$ by
      $\tilde{\phi}(v)=\sum_{\ell}\phi(\ell)^{\alpha_{\ell}}$.
    \item To be explicit, denote the mapping $\Z^n\to A$ by $\phi$.  Then part
      of sequence (\ref{exact-dual}) is
      \[
      \begin{array}{ccccc}
	A^*&\to&\mbox{Hom}(\Z^n,\C^{\times})&\approx&(\C^{\times})^n\\
	\chi&\mapsto&\chi\circ\phi&\mapsto&(\chi(a_1),\dots,\chi(a_n)).
      \end{array}
      \]
  \end{enumerate}
  We get an $n$-dimensional representation of $A^*$:
  \[
  \rho\colon A^*\to(\C^{\times})^n\to\mbox{GL}(\C^n)
  \]
  given by
  \[
  \rho(\chi)=\mbox{diag}(\chi(a_1),\dots,\chi(a_n)).
  \]
  In other words, the choice of generators for $A$ induces a homomorphism of
  $A^*$ into the group of invertible $n\times n$ matrices over $\C$.  (Every
  $n$-dimensional representation of $A^*$ over $\C$ is a direct sum of
  characters of $A^*$, i.e., of elements of $A^{**}\approx A$.  So this section
  can be regarded as saying something about representations of $A^*$, in
  general.)

  For each $z\in\C^n$, define the {\em orbit of $z$ under $\rho$} to be
  \[
    \mathcal{O}_{\rho}(z)= \{\rho(\chi)z:\chi\in A^*\}
    = \{(\chi(a_1)z_1,\dots,\chi(a_n)z_n):\chi\in A^*\}.
    \]
  \end{remark}
  We will assume that no coordinate of $z$ is zero, in which case by scaling
  coordinates of $\C^n$, we may assume for our purposes that $z=(1,\dots,1)$.
  Thus, we are interested in the orbit of the all-$1$s vector:
  \[
  \mathcal{O}=\{\rho(\chi):\chi\in A^*\}=
    \{(\chi(a_1),\dots,\chi(a_n)):\chi\in A^*\}.
  \]
\begin{definition}
  Let $I\subseteq R$ be an ideal.  Le $R_{\leq d}$ denote the vector space of
  polynomials in $R$ of degree at most $d$, and let $I_{\leq d}$ be the subspace
  $I\cap R_{\leq d}$.  The {\em affine Hilbert function} of $R/I$ is
  $H\colon\N\to\N$, given by
\[
H(d) := \dim_{\C}\, R_{\leq d}/I_{\leq d} = \dim_{\C}\, R_{\leq d}-\dim_{\C}\,
I_{\leq d}.
\]

\end{definition}
  \begin{thm}\label{thm:orbit ideal} Let $R=\C[x_1,\dots,x_n]$ and
    consider 
    \[
    I=\{f\in R: f(\mathcal{O})=0\},
    \]
    the ideal of polynomials
    vanishing on the orbit.  Then
    \begin{enumerate}
      \item\label{orbit ideal1} $I=I(\Lambda)=\myspan_{\C}\{x^u-x^v: u=v \bmod \Lambda\}$;
      \item\label{orbit ideal2}
	The affine Hilbert function of $R/I$ is given by
	\[
	H(d)=\#\left\{\textstyle\sum_{i=1}^nn_ia_i:n_i\geq0\mbox{ for all $i$ and $\sum_in_i\leq
	d$}\right\}.
	\]
    \end{enumerate}
  \end{thm}
  \begin{proof}
    This proof is due to the first author and Donna Glassbrenner.  It appears
    in~\cite{Campbell}.  Consider the matrix $M^{(d)}$ with rows indexed by
    $A^*$ and columns indexed by the monomials of $R_{\leq d}$ (arranged in
    lexicographical order so that $M^{(d)}$ is naturally nested in $M^{(d+1)}$)
    given by
    \[
    M^{(d)}_{\chi,x^u}=\prod_{i=1}^n\chi^{u_i}(a_i).
    \]
    Using the isomorphism 
    \begin{eqnarray*}
      A&\to& A^{**}\\
      a&\mapsto&\bar{a}
    \end{eqnarray*}
  where $\bar{a}(\chi):=\chi(a)$, we can write
    \[
    M^{(d)}_{\chi,x^u}=\prod_{i=1}^n\bar{a}_i^{u_i}(\chi)=\bar{a}^u(\chi)
    \]
    where $\bar{a}^u:=\prod_{i=1}^n\bar{a}_i^{u_i}\in A^{**}$.
    The $x^u$-th column of $M^{(d)}$ has entries $\bar{a}^u(\chi)$ as
    $\chi$ varies over $A^*$.   In other words, it is the list of all values of
    the function $\bar{a}^u$.  Thus, at least as far as linear algebra is
    concerned, the $x^u$-th column {\em is} $\bar{a}^u$.  Since distinct
    characters are linearly independent, it follows that any linear dependence
    relations are the result of columns that are equal.   
   
    Now, the $x^u$-th and $x^v$-th columns of $M^{(d)}$ are equal exactly when
    $\bar{a}^u=\bar{a}^v$ are equal.  This occurs exactly when
    $\sum_iu_ia_i=\sum_iv_ia_i$, which we write as $(u-v)\cdot a=0$ where
    $a:=(a_1,\dots,a_n)$.  In light of exact sequence~(\ref{exact}), this
    condition is equivalent to $u-v\in \Lambda$. 

    A vector $(\alpha_u)\in\ker M^{(d)}$ if and only if
    \[
    \sum_u\alpha_u\prod_{i=1}^n\chi^{u_i}(a_i)=0
    \]
    for all $\chi\in A^*$.  Thus, $(\alpha_u)\in\ker M^{(d)}$ if and only if
    the polynomial $p=\sum_u\alpha_ux^u$ vanishes on $\mathcal{O}$, i.e., $p\in
    I$.  Thus, elements of $I_{\leq d}$ correspond exactly with linear
    combinations among the columns of $M^{(d)}$.  As these relations are due to
    equality among columns, as already noted, part~\ref{orbit ideal1} follows.
    For part~\ref{orbit ideal2}, note that we have just shown that
    \[
    \dim I_{\leq d}=\dim R_{\leq d}-\mbox{rank}\ M^{(d)}.
    \]
    Since distinct characters are linearly independent, 
    \[
    \mbox{rank}\,M^{(d)}=\#\left\{\textstyle\sum_{i=1}^nn_ia_i:n_i\geq0\mbox{ for all $i$ and $\sum_in_i\leq
	d$}\right\}.
    \]
  \end{proof}

  Returning to the case of the toppling ideal, the exact sequence
  \[
  0\to \Z^n\stackrel{\tD}{\longrightarrow}\Z^n\to\mathcal{S}(G)\to0
  \]
  has the form of exact sequence~(\ref{exact}).  The generators $a_i$ are the
  configurations having exactly one grain of sand. 
  \begin{cor}~
  \be
\item\label{zpart1} The toppling ideal is the set of polynomials vanishing on an
  orbit $\mathcal{O}$ of a faithful representation of $\mathcal{S}(G)^*$.
   \item\label{zpart2} The set of zeros of the toppling ideal is the finite set,
     $\mathcal{O}$. It thus has the symmetry of $\mathcal{S}(G)^*$, which is
     isomorphic to the sandpile group.
   \item\label{zpart3}If $H_{G}$ is the affine Hilbert function for the toppling
     ideal, then $H_{G}(d)$ is the number of elements of $\Z^n/\tLap$ represented by
	configurations containing at most $d$
     grains of sand.  It is thus the number of superstable configurations of
     degree at most $d$ or, equivalently, the number of recurrent 
     configurations~$c$ such that
     \[
     \deg(c)\geq\deg(\cmax)-d.
     \]
  \ee
  \end{cor}
  \begin{proof}
    Part (\ref{zpart1}) follows directly from the first part of
    Theorem~\ref{thm:orbit ideal}.  For part~(\ref{zpart2}), since $\mathcal{O}$
    is a finite collection of points in $\C^n$, and $I(G)=I(\mathcal{O})$,
    it is a basic result of algebraic geometry that the set of zeros of $I(G)$
    is $\mathcal{O}$.  Part~(\ref{zpart3}) is immediate from the second part of
    Theorem~\ref{thm:orbit ideal} and the fact that $r$ is recurrent if and only if $\cmax-r$ is
    superstable.
  \end{proof}
  \begin{remark}
    Note that part~(\ref{zpart3}) also follows directly from
    Theorem~\ref{thm:normal basis}. 
  \end{remark}
  \subsubsection{Projective case.}\label{projective points}
  An ideal $J$ in $S=\C[x_1,\dots,x_{n+1}]$ is {\em homogeneous} if it has a set
  of homogeneous generators.  The set of zeros of $J$ is a subset of
  projective space:
  \[
  Z(J)=\{p\in\mathbb{P}^n:f(p)=0\mbox{ for all homogeneous $f\in J$}\}.
  \]
  The ring $S/J$ is graded by the integers: $(S/J)_{d}:=S_d/J_d$.
  \begin{definition}
  The {\em Hilbert function} of $S/J$ is $H\colon\N\to\N$, given by
  \[
  H(d):=\dim_{\C}(S/J)_d.
  \]
\end{definition}
Continuing with the notation from~\ref{affine case}, define the {\em
homogenization} of $\Lambda$ as
  \[
  \Lambda^h:=\left\{{{\ell}\choose{-\deg(\ell)}}\in\Z^{n+1}:\ell\in \Lambda\right\}.
  \]
  Consider the exact sequence
  \[
  0\to \Lambda^h\to\Z^{n+1}
  \stackrel{M}{\longrightarrow} A\oplus\Z\to0,
  \]
  where 
  \[
  M = \left(\begin{array}{cccc}a_1&\dots&a_n&0\\
  1&\dots&1&1\end{array}\right).
  \]
  Apply $\mbox{Hom}(\ \cdot\ ,\C^{\times})$ to get
  \begin{eqnarray*}
    1\to A^*\times\C^{\times}&\to&(\C^{\times})^{n+1}\to(\Lambda^h)^*\to 0\\
    (\chi,z)&\mapsto&(\chi(a_1)z,\dots,\chi(a_n)z,z)
  \end{eqnarray*}
  and the corresponding representation
  \begin{eqnarray*}
    A^*\times\C^{\times}&\to&\mbox{GL}(\C^{n+1})\\
    (\chi,z)&\mapsto&
    \mbox{diag}(\chi(a_1)z,\dots,\chi(a_n)z,z).
  \end{eqnarray*}
  The orbit of $(1,\dots,1)$ under this representation is
  \[
  \mathcal{O}^h=\{(\chi(a_1),\dots,\chi(a_n),1)\in\mathbb{P}^n:\chi\in
  A^*\}\subset\mathbb{P}^n.
  \]
  Thus, $\mathcal{O}^h$ is the {\em projective closure} of the orbit
  $\mathcal{O}$ from the previous section.
  \begin{thm}\label{thm:projective Hilbert}
    Let $a^h=(a_1,\dots,a_n,0)$.
    \begin{enumerate}
      \item The homogeneous ideal defining $\mathcal{O}^h$ is the lattice
	ideal for $\Lambda^h$, the homogenization of the lattice ideal for $\Lambda$:
	\[
	I^h=\{x^u-x^v:u=v\bmod \Lambda^h\}.
	\]
      \item The Hilbert function for $\mathcal{O}^h$ (i.e., the Hilbert function
	of $S/I^h$) is
	\[
	H(d)=\#\{b\cdot a^h\in A:\text{$b\in\N^{n+1}$ and $\deg(b)=d$}\},
	\]
      which is the same as the affine Hilbert function for $\mathcal{O}$.
    \end{enumerate}
  \end{thm}
  \begin{proof}
    Since $\mathcal{O}^h$ is the projective closure of $\mathcal{O}$, its ideal
    is $I^h$, the homogenization of the ideal defining $\mathcal{O}$,  which is
    given by $\{x^u-x^v:u=v\bmod\Lambda^h\}$.  The second part of the
    theorem follows from part~\ref{orbit ideal2} of Theorem~\ref{thm:orbit ideal}
    and the isomorphism of vector spaces
    \begin{align*}
      S_d&\to R_{\leq d}\\
      f&\mapsto f|_{x_{n+1}=1},
    \end{align*}
    with inverse $g(x_1,\dots,x_n)\mapsto
    x_{n+1}^d\,g(x_1/x_{n+1},\dots,x_n/x_{n+1})$.
  \end{proof}
  \begin{cor} Suppose $\Lambda=\tLap$, the reduced Laplacian lattice of $G$,
    and that $\Delta(v_{n+1})\in\myspan_{\Z}\{\Delta(v_i):1\leq i\leq
    n\}$ so that $\Lambda^h=\Lap$, the full Laplacian lattice (see the
    comments preceding Proposition~\ref{prop:homog}). 
    \be
     \item The homogenization of the toppling ideal is the ideal generated by all
    homogeneous polynomials vanishing on an orbit $\mathcal{O}^h$ of a
    faithful representation of $(\Z^{n+1}/\Lap)^*$.
     \item The set of zeros of the homogenization of the toppling ideal is the
    finite set $\mathcal{O}^h$ having the symmetry of
    $\mathcal{S}(G)^*$.
    \ee
  \end{cor}
  
  \subsubsection{The $h$-vector.}  Let $\Delta H_G$
  denote the first differences of the affine Hilbert function of a sandpile
  graph $G$.  So $\Delta H_G(d):=H_G(d)-H_G(d-1)$.  By Theorem~\ref{thm:normal
  basis}, the value of $\Delta H_G(d)$ is the number of superstable
  configurations of degree $d$. 
  \begin{definition}\label{def:h-vector}
    Let $h_d:=\Delta H_G(d)$. The {\em postulation number} for $G$ is the
    largest integer $\ell$ such that $h_{\ell}\neq0$.  The $h$-vector for $G$ is
    $h=(h_0,\dots,h_{\ell})$.  The {\em Hilbert-Poincar\'e series} for $G$ is
    $P_G(y)=\sum_{i=0}^{\ell}h_iy^i$.
  \end{definition}
  \begin{example}\label{example:h-vector}
    Continuing Example~\ref{example:sandpile graph}, the $h$-vector for the
    sandpile graph in Figure~\ref{fig:sandpile graph} is $(1,3,6,7,4)$.
  \end{example}

  Let the vertices of $G$ be $\{v_1,\dots,v_{n+1}\}$ with $v_{n+1}$ as the sink,
  as usual.  Let $I^h\subseteq I_h\subset S=\C[x_1,\dots,x_{n+1}]$ be the
  homogenization of the the toppling ideal and the homogeneous toppling ideal
  for $G$, respectively.  These two ideals are identical when the hypothesis of
  Proposition~\ref{prop:homog} is satisfied.  In any case, their zero-sets
  satisfy $Z(I^h)\supseteq Z(I_h)$.  Pick a linear polynomial $f\in
  S$ that does not vanish at any point of $Z(I^h)$.  For instance, we could
  take $f=x_i$ for any $i$.  Multiplication by $f$ gives rise to the commutative
  diagram with exact rows
\[
\xymatrix{
0\ar[r]&S/I^h\ar[r]^{\cdot f}\ar@{>>}[d]&
S/I^h\ar[r]\ar@{>>}[d]&S/(I^h+(f))\ar[r]\ar@{>>}[d]&0\\
0\ar[r]&S/I_h\ar[r]^{\cdot f}&S/I_h\ar[r]&S/(I_h+(f))\ar[r]&0.
}
\]
By this diagram and Theorem~\ref{thm:projective Hilbert}, we have the following
relations among the first differences of Hilbert functions:
\begin{align}\label{eqn:first differences}
  h_d&=\Delta H_G(d)=\Delta H_{S/I^h}(d)=H_{S/(I^h+(f))}(d)\\
  &\geq H_{S/(I_h+(f))}(d)=\Delta H_{S/I_h}(d).\nonumber
\end{align}
\subsubsection{The Tutte polynomial.}
Now let $G=(V,E)$ be any (weighted, directed) graph, and $e\in E$.  Let
$G- e$ denote the graph obtained from $G$ by replacing $\wt(e)$ by
$\wt(e)-1$.  In other words, imagine the endpoints $e^{-}$ and $e^{+}$ attached
by $\wt(e)$ edges, and remove one of these edges to obtain $G-e$.  In
particular, if $\wt(e)=1$, this amounts to removing the edge $e$.  Let $G/e$
denote the graph obtained from $G$ by identifying the endpoints of $e$ and
lowering the weight of $e$ by one.  We refer to these two operations on $G$ as
{\em deletion} and {\em contraction}.  The edge $e$ is called a {\em bridge} if
$G-e$ has more components than $G$.
\begin{definition}
  Let $G$ be an undirected, weighted graph.
  Define the {\em Tutte polynomial}, $T_G(x,y)$ for $G$ recursively, as follows.  If $E$ consists of
  $i$ bridges, $j$ loops, and no other edges, then
  \[
   T_G(x,y):=x^iy^j.
  \]
  In particular, $T_G=1$ if $G$ has no edges.
  Otherwise, if $e\in E$ is neither a bridge nor a loop, then
  \[
  T_G:=T_{G-e}+T_{G/e}.
  \]
\end{definition}
It turns out the the Tutte polynomial is well-defined, independent of 
choices for deletions and contractions.  It is well-known that
\[
C_G(x)=(-1)^{\#V-\kappa(G)}x^{\kappa(G)}T_G(1-x,0),
\]
where $C_G$ is the
chromatic polynomial of $G$ and $\kappa(G)$ is the number of components of $G$.
The following result relates other specializations of the Tutte polynomial
to the algebraic geometry of sandpiles.
\begin{thm}[Merino \cite{Merino}]\label{thm:Merino}
  Let $G$ be an undirected sandpile graph with
  postulation number $\ell$.
  Then
  \[
  T_G(1,y)=\sum_{i=0}^{\ell}h_{\ell-i}y^i.
  \]
\end{thm}
\begin{cor}\label{cor:tutte}
  Let $G$ be as in Theorem~\ref{thm:Merino}.  Then
  \begin{enumerate}
    \item the Hilbert-Poincar\'e series for $G$ is $y^{\ell}\,T(1,1/y)$;
    \item\label{merino2} if $d$ is the degree of the maximal stable
      configuration on $G$, then $y^{d-\ell}\,T(1,y)$ is the generating function
      for the recurrent configurations of $G$ (by degree);
    \item $T_G(1,1)$ is the order of the sandpile group of $G$;
    \item $T_G(1,0)$ is the number of maximal superstable (or the number of
      minimal recurrent) configurations of $G$.
  \end{enumerate}
\end{cor}
\begin{proof}
  These results follow immediately from Theorem~\ref{thm:Merino}.
  Part~(\ref{merino2}) uses the fact that $c$ is superstable if and only if
  $\cmax-c$ is recurrent.
\end{proof}
\begin{example}
  Figure~\ref{fig:tutte} shows the construction of the Tutte polynomial of a
  graph~$G$.  We have $T(1,y)=4+3y+y^2$ and $T(1,1)=8$.  Fixing the
  southern-most vertex of $G$ as the sink gives a sandpile graph with $h$-vector
  $(1,3,4)$ and sandpile group of order $8$.
  \begin{figure}[ht] 
\begin{tikzpicture}[scale=0.8]
\SetVertexMath
\GraphInit[vstyle=Art]
\SetUpVertex[MinSize=3pt]
\SetVertexNoLabel
\tikzset{VertexStyle/.style = {%
shape = circle,
shading = ball,
ball color = black,
inner sep = 1.5pt
}}
\SetUpEdge[color=black]


\draw (0,10) circle(1.4);
\Vertex[x=0,y=11]{x}
\Vertex[x=-0.866,y=10]{y}
\Vertex[x=0.866,y=10]{z}
\Vertex[x=0,y=9]{s}
\Edge[](x)(y)
\Edge[](x)(z)
\Edge[style={dashed}](y)(z)
\Edge[](s)(y)
\Edge[](s)(z)

\draw [ultra thick] (-1.300,9.220) -- (-3.700,7.780);
\draw [ultra thick] (1.200,9.100) -- (3.040,7.720);
\draw (0,8) node{$G$};

\draw (-5,7) circle(1.4);
\Vertex[x=-5,y=8]{x}
\Vertex[x=-5.866,y=7]{y}
\Vertex[x=-4.134,y=7]{z}
\Vertex[x=-5,y=6]{s}
\Edge[style=dashed](x)(y)
\Edge[](x)(z)
\Edge[](s)(y)
\Edge[](s)(z)

\draw [ultra thick] (-5.840,5.740) -- (-6.16,5.26);
\draw [ultra thick] (-4.064,5.83) -- (-3.416,5.02);

\draw (-7,4) circle(1.4);
\Vertex[x=-7,y=5]{x}
\Vertex[x=-7.866,y=4]{y}
\Vertex[x=-6.134,y=4]{z}
\Vertex[x=-7,y=3]{s}
\Edge[](x)(z)
\Edge[](s)(y)
\Edge[](s)(z)

\draw (-8,2.3) node {$x^3$};
\draw [very thick] (-8.3,1.9) -- (-7.9,1.9);

\draw (-2.6,4) ellipse(1.0 and 1.4);
\Vertex[x=-2.8,y=5]{x}
\Vertex[x=-1.934,y=4]{z}
\Vertex[x=-2.8,y=3]{s}
\Edge[](x)(z)
\Edge[style=dashed](x)(s)
\Edge[](s)(z)

\draw [ultra thick] (-3.528,3.13) -- (-4.872,1.870);
\draw [ultra thick] (-2.475,2.5) -- (-2.425,1.9);

\draw (-5.8,1) ellipse(1.0 and 1.4);
\Vertex[x=-6,y=2]{x}
\Vertex[x=-5.134,y=1]{z}
\Vertex[x=-6,y=0]{s}
\Edge[](x)(z)
\Edge[](s)(z)

\draw (-7,-0.7) node {$x^2$};
\draw [very thick] (-7.3,-1.1) -- (-6.9,-1.1);

\draw (-2.35,1) ellipse(1.2 and 0.8);
\Vertex[x=-3,y=1]{s}
\Vertex[x=-1.7,y=1]{y}
\Edge[style={dashed, bend left=30}](s)(y)
\Edge[style={bend left=30}](y)(s)

\draw [ultra thick] (-2.8945,0.1750) -- (-3.571,-0.850);
\draw [ultra thick] (-2.0525,0.125) -- (-1.789,-0.650);

\draw (-4,-1.5) ellipse(1.0 and 0.6);
\Vertex[x=-4.5,y=-1.5]{s}
\Vertex[x=-3.5,y=-1.5]{y}
\Edge[](s)(y)

\draw (-5.5,-2.2) node {$x$};
\draw [very thick] (-5.7,-2.55) -- (-5.4,-2.55);

\draw (-1.5,-1.5) circle(0.8);
\Vertex[x=-1.5,y=-1.9]{s}
\Loop[dist=1.5cm,dir=NO,style={thick}](s)

\draw (-0.5,-2.4) node {$y$};
\draw [very thick] (-0.7,-2.75) -- (-0.4,-2.75);

\draw (4,7) ellipse(1.0 and 1.4);
\Vertex[x=4,y=8]{x}
\Vertex[x=4,y=7]{y}
\Vertex[x=4,y=6]{s}
\Edge[style={bend left=50}](x)(y)
\Edge[style={dashed,bend right=50}](x)(y)
\Edge[style={bend left=50}](s)(y)
\Edge[style={bend right=50}](s)(y)

\draw [ultra thick] (3.260,5.890) -- (2.740,5.110);
\draw [ultra thick] (4.740,5.890) -- (5.260,5.110);

\draw (2,4) ellipse(1.0 and 1.4);
\Vertex[x=2,y=5]{x}
\Vertex[x=2,y=4]{y}
\Vertex[x=2,y=3]{s}
\Edge[](x)(y)
\Edge[style={dashed, bend left=50}](s)(y)
\Edge[style={bend right=50}](s)(y)

\draw [ultra thick] (1.496,2.656) -- (1.280,2.080);
\draw [ultra thick] (2.444,2.592) -- (2.580,2.144);

\draw (6,4) ellipse(1.0 and 1.4);
\Vertex[x=6,y=4]{y}
\Vertex[x=6,y=3]{s}
\Loop[dist=1.5cm,dir=NO,style={thick}](y)
\Edge[style={dashed, bend left=50}](s)(y)
\Edge[style={bend right=50}](s)(y)

\draw [ultra thick] (5.570,2.624) -- (5.420,2.144);
\draw [ultra thick] (6.504,2.656) -- (6.708,2.112);

\draw (0.8,0.8) ellipse(0.8 and 1.4);
\Vertex[x=0.8,y=1.8]{x}
\Vertex[x=0.8,y=0.8]{y}
\Vertex[x=0.8,y=-0.2]{s}
\Edge[](x)(y)
\Edge[](s)(y)
\Edge[](s)(y)

\draw (0.8,-1.2) node {$x^2$};
\draw [very thick] (0.5,-1.6) -- (0.9,-1.6);

\draw (3.0,0.8) ellipse(0.8 and 1.4);
\Vertex[x=3,y=1.8]{x}
\Vertex[x=3,y=0.8]{y}
\Edge[](x)(y)
\Loop[dist=1.5cm,dir=SO,style={thick}](y)

\draw (3,-1.27) node {$xy$};
\draw [very thick] (2.7,-1.6) -- (3.3,-1.6);

\draw (5.0,0.8) ellipse(0.8 and 1.4);
\Vertex[x=5,y=0.8]{y}
\Vertex[x=5,y=-0.2]{s}
\Loop[dist=1.5cm,dir=NO,style={thick}](y)
\Edge[](s)(y)

\draw (5,-1.27) node {$xy$};
\draw [very thick] (4.7,-1.6) -- (5.3,-1.6);

\draw (7.2,0.8) ellipse(0.8 and 1.4);
\Vertex[x=7.2,y=0.8]{s}
\Loop[dist=1.5cm,dir=NO,style={thick}](s)
\Loop[dist=1.5cm,dir=SO,style={thick}](s)

\draw (7.2,-1.20) node {$y^2$};
\draw [very thick] (6.9,-1.6) -- (7.3,-1.6);

\draw (0.5,-4) node {$T_G(x,y)=x+2x^2+x^3+(1+2x)y+y^2$};
\end{tikzpicture}
\caption{The Tutte polynomial of $G$.}\label{fig:tutte}
\end{figure}
\end{example}
  \subsubsection{Cayley-Bacharach property.}
  Let $X\subset\mathbb{P}^n$ be a finite set of points in projective space, and let
  $I(X)\subset S:=\C[x_1,\dots,x_{n+1}]$ be
  the ideal generated by the homogeneous polynomials
  vanishing on $X$.  If $H_X$ is the Hilbert function of $S/I(X)$, then $H_X(d)$ is
  the number of linear conditions placed on the coefficients of a general
  homogeneous polynomial of degree $d$ in $S$ by requiring the polynomial to vanish
  on the points of $X$.  Thus, $H_X$ is a monotonically increasing function which
  is eventually constant at $|X|$.  The first value at which $H_X$ takes the
  value $|X|$ is called the {\em postulation number} for $X$.

  \begin{definition}\label{CB} A finite set of points $X\subset\mathbb{P}^n$ is {\em
    Cayley-Bacharach} if it satisfies one of the following equivalent
    conditions.
    \begin{enumerate}
      \item\label{truncate} For each $p\in X$, and for each $d\in\N$,
	\begin{eqnarray*}
	  H_{X\setminus\{p\}}(d) = \min\{H_X(d),|X|-1\}.
	\end{eqnarray*}
      \item Every homogeneous polynomial with degree less than the postulation
	number for~$X$ and vanishing on all but one point of~$X$ must vanish on
	all of~$X$.
    \end{enumerate}
  \end{definition}

  \begin{prop}\label{prop:CB-yes}
    The set of zeros of the homogeneous toppling ideal is Cayley-Bacharach.
  \end{prop}
  \begin{proof}
    By Proposition~1.14 of \cite{GKR}, for any finite set of points, $X$, there
    is always at least one point $p$ for which condition (\ref{truncate}) of
    Definition~\ref{CB} holds.  However,
    in our case, $X$ is the orbit of a linear representation of the sandpile
    group.  Thus, given any two points $p,q\in X$, there is a linear change of
    coordinates of $\mathbb{P}^n$ sending $p$ to $q$.  A linear change of
    coordinates does not change the Hilbert function.  Hence, condition (\ref{truncate})
    holds for all points of $X$. 
  \end{proof}

  \begin{remark}
    Let $X$ be the set of zeros of a homogeneous toppling ideal and define the first differences
    of its Hilbert function by $\Delta H_X(d) =
    H_X(d)-H_X(d-1)$ for all $d\in \Z$.  
    It follows from results in \cite{GKR} and the fact that $X$ is
    Cayley-Bacharach, that if the last nonzero value of
    $\Delta_X$ is $m$, then there is a collection of $m$ points $Y\subset X$
    such that $X\setminus Y$ is Cayley-Bacharach.  Moreover, if $m=1$, then
    every subset of $X$ of size $|X|-1$ is Cayley-Bacharach.
  \end{remark}

  \section{Resolutions}\label{section:resolutions}
  In this section, we consider the minimal free resolution of the homogeneous
  toppling ideal, summarizing some of the results in~\cite{Wilmes}.  For further
  work on resolutions of toppling ideals, see~\cite{Madhu}.  First, we recall
  the language of divisors on graphs from \cite{Baker} (extended to directed
  multigraphs).  Let $G$ be a directed multigraph as in \S\ref{sandpiles}.  The
  free Abelian group $\Z V$ on the vertices of~$G$ is denoted $\divisor(G)$,
  and its elements are called {\em divisors}.  The {\em degree} of a divisor
  $D=\sum_{v\in V}D_v\,v\in\mbox{div}(G)$, is $\deg(D):=\sum_{v\in V}D_v$.  A
  divisor is {\em principal} if it is in the Laplacian lattice $\Lap$, defined
  in \S\ref{sandpiles}.  Divisors $D$ and $D'$ are {\em linearly equivalent},
  written $D\sim D'$, if $D-D'$ is principal.  Note that linearly equivalent
  divisors must have the same degree.  The group of divisors modulo linear
  equivalence is the {\em class group} of $G$, denoted $\mbox{Cl}(G)$.  In the
  case where $G$ is an Eulerian sandpile graph, using the notation of
  Proposition~\ref{prop:sink}, there is an isomorphism
  \begin{align*}
    \mbox{Cl}(G)&\to\Z\oplus\Z V_0/\Lap\\
    D&\mapsto(\deg D,D-(\deg D)s),
  \end{align*}
  where $\Z V_0/\Lap$ is isomorphic to the sandpile group $\sand(G)$.

  We will usually denote a divisor class $[D]\in\mbox{Cl}(G)$ by just $D$,
  choosing a representative divisor, when the context is clear.  A divisor
  $D=\sum_{v\in V}D_v\,v$ is {\em effective} if $D\geq 0$.  The collection of
  all effective divisors linearly equivalent to $D$ is called the {\em
  (complete) linear system} for $D$ and denoted $|D|$;  it only depends on the
  divisor class of $D$.  The {\em support} of a divisor $D$ is
  $\mathrm{supp}(D):=\{v\in V:D_v\neq0\}$.

  One might think of a divisor as an assignment of money to each vertex, with
  negative numbers denoting debt.  Just as with configurations in the sandpile
  model, the Laplacian determines firing rules by which vertices can lend to or
  borrow from neighbors.  Two divisors are linearly equivalent if one can be
  obtained from the other by a sequence of vertex lendings and borrowings.  The
  complete linear system corresponding to a divisor is nonempty if there is a
  way for vertices to lend and borrow, resulting in no vertex being in debt.

\subsection{Riemann-Roch.}
  To recall the graph-theoretic Riemann-Roch theorem of~\cite{Baker}, let
  $G=(V,E)$ be an undirected graph.  Define the {\em genus} of $G$ to be
  \[
  g := \#E-\#V+1.
  \]
  Define the {\em dimension} of the linear system $|D|$ for a divisor $D$ on $G$ to be
  \[
  r(D):=\max\{k\in\Z: \text{$|D-E|\neq\emptyset$ for all $E\geq0$ with
  $\deg(E)=k$}\},
  \]
  with $r(D):=-1$ if $|D|=\emptyset$.  Note that $r(D)$ depends only on
  the divisor class of~$D$.  Define the {\em maximal stable divisor},
  \[
  \Dmax:=\sum_{v\in V}(\deg(v)-1)v,
  \]
  and the {\em canonical divisor},
  \[
  K:=\Dmax-\vec{1}=\sum_{v\in V}(\deg(v)-2)v.
  \]
  \begin{thm}[Riemann-Roch Theorem \cite{Baker}]\label{thm:rr}  Let $G$ be an undirected graph.  For all $D\in\divisor(G)$, 
    \[
    r(D)-r(K-D)=\deg(D)+1-g.
    \]
  \end{thm}
  \begin{remark}
    This Riemann-Roch theorem is generalized in~\cite{Madhu} to the context of
    certain monomial ideals, relating it to Alexander duality in combinatorial
    commutative algebra.  From that point of view, the relevant monomial ideal
    for us is the ideal generated by the leading terms of a homogeneous
    toppling ideal with respect to a sandpile monomial ordering.  It is noted
    that these monomial ideals are studied by Postnikov and Shapiro
    in~\cite{Postnikov}.
  \end{remark}
\subsection{Resolutions and Betti numbers.}
  Let $G$ be an arbitrary directed multigraph.  Identify the vertices of $G$
  with the set $\{1,\dots,n+1\}$, with $n+1$ being the sink.  The polynomial
  ring $S=\C[x_1,\dots,x_{n+1}]$ is graded by the class group by letting the
  degree of a monomial $x^D$ be $D\in\mbox{Cl}(G)$.  For each
  $D\in\mbox{Cl}(G)$, let $S_D$ be the $\C$-vector space generated by the
  monomials of degree $D$, and define the {\em twist},~$S(D)$, by letting
  $S(D)_F:=S_{(D+F)}$ for each $F\in\mbox{Cl}(G)$.

  Let $I:=I_h(G)$ be the homogeneous toppling ideal.  A {\em free
  resolution} of $I$ is an exact sequence
 \[ 
 0\leftarrow I
 \stackrel{\phi_0}{\longleftarrow}
 F_1 
 \stackrel{\phi_1}{\longleftarrow}
 F_2
\leftarrow
\dots
 \stackrel{\phi_r}{\longleftarrow}
 F_r
  \leftarrow
  0, 
\]
where each $F_i$ is a free $\mathrm{Cl}(G)$-graded $S$-module, i.e.,
\[
F_i=\bigoplus_{D\in\mathrm{Cl}(G)}S(-D)^{\beta_{i,D}}
\]
for some nonnegative integers $\beta_{i,D}$, and where each $\phi$ preserves
degrees.  The {\em length} of the resolution is $r$.  A free resolution is {\em
minimal} if each of the $\beta_{i,D}$ is the minimum possible from among all
free resolutions of $I$.  In this case, the $\beta_{i,D}$ are called the {\em
Betti numbers} of $I$.  For instance, $\beta_{1,D}$ is the number of polynomials
of degree $D$ in a minimal generating set for $I$.  We also define the $i$-th
{\em coarsely graded Betti number} of $I$ by $\beta_i = \sum_{D \in
\mathrm{Cl}(D)}\beta_{i,D}$. 

The following theorem states a well-known fact about resolutions of sets of
points in projective space (the Cohen-Macaulay property).
\begin{prop}
  The length of the minimal free resolution of the homogeneous toppling ideal is
  $n$, the number of nonsink vertices.
\end{prop}

\begin{example}
\label{ex:resolution}
\begin{figure}[ht] 
\begin{tikzpicture}[scale=1.7]
\SetVertexMath
\GraphInit[vstyle=Art]
\SetUpVertex[MinSize=3pt]
\SetVertexLabel
\tikzset{VertexStyle/.style = {%
shape = circle,
shading = ball,
ball color = black,
inner sep = 1.5pt
}}
\SetUpEdge[color=black]
\Vertex[LabelOut,Lpos=270, Ldist=.3cm,x=0,y=0]{v_1}
\Vertex[LabelOut,Lpos=270, Ldist=.3cm,x=1,y=0]{v_2}
\Vertex[LabelOut,Lpos=270, Ldist=.3cm,x=2,y=0]{v_3}
\Vertex[LabelOut,Lpos=270, Ldist=.3cm,x=3,y=0]{v_4}
\Edge[style={-triangle 45, bend left=40}](v_1)(v_2)
\Edge[style={-triangle 45, bend left=40}](v_2)(v_1)
\fill[color=white] (0.45,0.2) circle (0.1cm);
\draw (0.45,0.2) node{{\small 5}};
\fill[color=white] (0.6,-0.2) circle (0.1cm);
\draw (0.6,-0.2) node{{\small 1}};
\Edge[style={-triangle 45, bend left=30}](v_2)(v_3)
\Edge[style={-triangle 45, bend left=40}](v_3)(v_2)
\fill[color=white] (1.4,0.15) circle (0.1cm);
\draw (1.4,0.15) node{{\small 1}};
\fill[color=white] (1.6,-0.2) circle (0.1cm);
\draw (1.6,-0.2) node{{\small 4}};
\Edge[style={-triangle 45}](v_3)(v_4)
\fill[color=white] (2.4,0.0) circle (0.1cm);
\draw (2.4,0.0) node{{\small 1}};
\Edge[style={-triangle 45, bend right=70}](v_3)(v_1)
\fill[color=white] (1,0.58) circle (0.1cm);
\draw (1,0.58) node{{\small 1}};
\end{tikzpicture}
\caption{A Gorenstein sandpile graph $G$ with sink $v_4$.}
\label{fig:gor}
\end{figure}
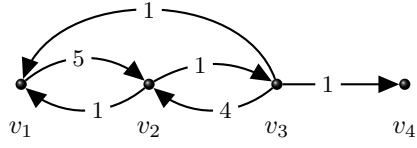

Let $G$ be as in Figure \ref{fig:gor} and let $I = I(G)^h$. Then 
\[
\xymatrix{
0 & I \ar[l] \restore
& S^5 \ar[l]_-{\phi_0} \save -<0pt,35pt>*{\begin{matrix}
\scriptstyle 0101 \\
\scriptstyle 0110 \\ \scriptstyle 1010 \\ \scriptstyle 1100 \\ \scriptstyle 0011
\end{matrix}} \restore & S^5
\ar[l]_-{\phi_1} \save -<0pt,35pt>*{\begin{matrix} \scriptstyle 1011 \\
\scriptstyle 1101 \\
\scriptstyle 1110 \\ \scriptstyle 2010 \\ \scriptstyle 0111 \end{matrix}} \restore & S \ar[l]_-{\phi_2}
\save -<0pt,10pt>*{\scriptstyle 1022} \restore & 0 \ar[l] 
}
\]
is a minimal free resolution for $I$, where the $\phi_i$ are given by
\begin{eqnarray*}
\phi_0 & = & \begin{bmatrix} x_3^2 - x_2x_4 & x_2x_3 - x_1x_4 & x_2^2 - x_1x_3
& x_1x_2 - x_4^2 & x_1^2 - x_3x_4 \end{bmatrix} \\ 
\phi_1 & = & \begin{bmatrix}  x_2 &  x_1 &  0   &  x_4 &  0   \\
                -x_3 & -x_2 &  x_1 &  0   & -x_4 \\
                 x_4 &  x_3 &  0   &  x_1 &  0   \\
                 0   &  0   & -x_3 & -x_2 &  x_1 \\
                 0   &  0   &  x_4 &  x_3 & -x_2 \end{bmatrix} \\ 
\phi_2 & = & \begin{bmatrix} x_1^2 - x_3x_4 \\ -x_1x_2+x_4^2\\ 
  -x_2^2+x_1x_3
\\ x_2x_3 - x_1x_4 \\ x_3^2 - x_2x_4 \end{bmatrix}.
\end{eqnarray*}
The grading of the $S$-modules is
indicated below each of them. For example, the last $S$-module is
$S(-(1,0,2,2))$.
\end{example}

The Betti numbers of $I$ may be understood topologically. For
$D\in\mathrm{Cl}(G)$, define the simplicial complex $\Delta_D$ on the vertices
of $G$ by $W\in\Delta_D$ if and only if $W\subseteq\mathrm{supp}(E)$ for some
$E\in|D|$. The following version of Hochster's formula appeared as Lemma 2.1 of
\cite{peeva}.

\begin{thm} \label{betti homology}
  The Betti number $\beta_{i,D}$ is the dimension of the $(i-1)$-th reduced homology
  group $\widetilde{H}_{i-1}(\Delta_D;\C)$ as a $\C$-vector space.
\end{thm}

\begin{example}\label{example:hochster1}
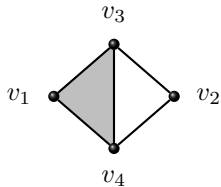
\begin{figure}[ht] 
\begin{tikzpicture}[scale=0.8]
\SetVertexMath
\GraphInit[vstyle=Art]
\SetUpVertex[MinSize=3pt]
\SetVertexLabel
\tikzset{VertexStyle/.style = {%
shape = circle,
shading = ball,
ball color = black,
inner sep = 1.5pt
}}
\SetUpEdge[color=black]
\draw[fill=gray!50] (-1,0) -- (0,0.866) -- (0,-0.866);
\Vertex[LabelOut,Lpos=180, Ldist=.1cm,x=-1,y=0]{v_1}
\Vertex[LabelOut,Lpos=90, Ldist=.1cm,x=0,y=0.866]{v_3}
\Vertex[LabelOut,Lpos=270, Ldist=.1cm,x=0,y=-0.866]{v_4}
\Vertex[LabelOut,Lpos=0, Ldist=.1cm,x=1,y=0]{v_2}
\Edges(v_1,v_3,v_2,v_4,v_1)
\Edge[](v_3)(v_4)
\end{tikzpicture}
\caption{The simplicial complex $\Delta_D$ for Example~\ref{example:hochster1}.}\label{fig:hochster1}
\end{figure}
Let $G$ again be as in Figure~\ref{fig:gor}. For $D = v_1 + v_3 + v_4$, we saw
in Example \ref{ex:resolution} that $\beta_{2,D} = 1$. We have
\[
|D| = \{D, v_2 + 2 v_3, 3 v_1, 2 v_2 + v_4\},
\]
so the simplicial complex $\Delta_D$ is as pictured in
Figure~\ref{fig:hochster1}.  Note
$\dim_\C\widetilde{H}_{1}(\Delta_D;\C)~=~1$, as asserted by Hochster's
formula.
\end{example}

\subsection{Minimal recurrents.}
Again specialize to the case of an undirected graph $G$.  As part of the
Riemann-Roch theory, one defines the {\em non-special divisors} on $G$ to be
\[
\mathcal{N}:=\{D\in\divisor(G): \text{$\deg(D)=g-1$ and $|D|=\emptyset$}\}.
\]
By the Riemann-Roch theorem, if $\deg(D)>g-1$, then $|D|\neq\emptyset$.  So the
nonspecial divisors are the divisors of maximal degree having empty linear
system. 

Fix $s\in V$ and consider the sandpile graph $G=(V,E,s)$.  A recurrent
configuration $c$ on $G$ is {\em minimal} if $c-v$ is not recurrent for any
nonsink vertex $v$.  It is well-known that (since $G$ is undirected) the minimal
recurrent configurations are exactly the recurrent configurations of minimal
degree, namely of degree $\#E-\deg(s)$.  (This result follows from Dhar's
algorithm (cf.~\S\ref{subsect:burning} and the proof of
Theorem~\ref{thm:loopy}).)  Similarly, one says that a superstable configuration
$c$ is a {\em maximal} if $c+v$ is not superstable for any nonsink vertex $v$.
By Corollary~\ref{cor:duality}, it follows that the maximal superstable
configurations are exactly those of degree $g$.  

We say that a divisor $D$ on $G$ is {\em unstable} if $D_v\geq\deg(v)$ for some
$v\in V$ and that~$D$ is {\em alive} if there is no stable divisor in
$|D|$.  Further, $D$ is {\em minimally alive} if for all $v\in V$, we have that
$D-v$ is not alive.  It is shown in~\cite{Wilmes} that a divisor~$D$ is alive if
and only if $D\sim c+k\,s$ for some recurrent configuration $c$ and some
$k\geq\deg(s)$, and $D$ is minimally alive if and only if $D\sim c+\deg(s)\,s$
for some minimal recurrent configuration $c$.

It is shown in~\cite{Baker} that a set of representatives for the
distinct divisor classes of the non-special divisors is 
\[
\{c-s: \text{$c$ a maximal superstable configuration}\}.
\]
Thus, the non-special divisor classes are given, essentially, by the maximal
$G$-parking functions.

Suppose that $\nu$ is a nonspecial divisor.  We may assume $\nu=c-s$ for some maximal
superstable configuration $c$.  Then
\[
\Dmax-\nu=(\cmax-c)+\deg(s)\,s.
\]
Since $\cmax-c$ is a minimal recurrent configuration, $\Dmax-\nu$ is
minimally alive.  Similarly, one may show that if $D'$ is a minimally alive
divisor, then $\Dmax-D'$ is nonspecial.  Thus, on an undirected graph there is a
bijective correspondence between: minimal recurrent configurations, maximal
superstable configurations, maximal $G$-parking functions, acyclic orientations
with $s$ as the unique source vertex (cf.~Theorem~\ref{thm:acyclic
orientations}), minimally alive divisors, and non-special divisors. In
particular, the cardinality of these sets does not depend on the choice of sink.

The following is Theorem 3.10 of \cite{Wilmes}.  The proof is included here for
the sake of completeness.

\begin{thm}\label{thm:last betti number} 
  Let $G$ be an undirected graph and $D\in\divisor(G)$.   Let $r=\#V-1$, the
  length of a minimal free resolution for $G$.  Then the highest
  nonzero Betti number,~$\beta_{r}$, is the number of minimal recurrent
  configurations on $G$.  We have
  \[
  \beta_{r,D}\neq0
  \]
  if and only if $D$ is minimally alive (in which case $\deg(D)=\#E$).
\end{thm}
\begin{proof}
First note that by Theorem \ref{thm:rr}, a divisor $\nu$ is nonspecial if and
only if $K - \nu$ is nonspecial. Indeed, if $\deg(\nu) = g - 1$, then
\[
\deg(K - \nu) = (2g - 2) - g - 1 = g - 1
\]
so that Theorem $\ref{thm:rr}$ gives $r(K - \nu) = r(\nu)$.

By Theorem \ref{betti homology}, we have $\beta_{r,D} =
\dim_\C\widetilde{H}_{r-1}(\Delta_D;\C)$. Since for any $D\in\divisor(G)$ the
simplicial complex $\Delta_D$ has $\#V$ vertices, $\beta_{r,D} \ne 0$ if and
only if $\Delta_D$ is the boundary of an $r$-simplex. Thus, $\beta_{r,D} \ne 0$
if and only if $\beta_{r,D} = 1$, or equivalently: (i) no $E \in |D|$ has full
support, and (ii) for every $v \in V$ there is some $E \in |D|$ with
$V\setminus\{v\}\subseteq\supp(E)$.

Suppose $D$ is minimally alive. Then $\Dmax - D$ is nonspecial by the discussion
preceding the statement of the theorem. Let $\nu = K - (\Dmax - D)$, so that
$\nu$ is also nonspecial.  In particular, $|D - \vec{1}|= |\nu| = \emptyset$, so
no divisor $E \in |D|$ has full support. Now fix $v \in V$ and let $F = D -
\vec{1} + v$.  Note that a divisor $E \in |D|$ satisfies $V\setminus\{v\}
\subseteq \supp(E)$ if and only if $E - \vec{1} + v \in |F|$.  So to complete the
proof that $\beta_{r,D} \ne 0$ it suffices to show that $|F| \ne \emptyset$.
Note that $\deg(F) = g$. Since $K - F + v = \Dmax - D$, we have $K - F + v$
nonspecial, and it follows that $r(K - F) = -1$.  Thus, by Theorem \ref{thm:rr},
we have $r(F) = 0$ as desired.  Hence, $D$ satisfies (i) and (ii).

On the other hand, suppose $D$ satisfies (i) and (ii) above, and let $\nu = D -
\vec{1}$. Then $|\nu| = \emptyset$ follows from (i), and therefore $\Dmax -
\nu$ is alive. On the other hand, for every $v \in V$ we have from (ii) that
$|\nu + v| \ne \emptyset$, whence $(\Dmax - \nu) - v$ is not alive. Thus,
$\Dmax - \nu$ is minimally alive, so that $\nu$ is nonspecial. But then $K -
\nu = \Dmax - D$ is also nonspecial, implying $D$ is minimally alive.
\end{proof}

\begin{example}\label{example:genus2}
  We summarize many of the results of this paper using the graph~$G$ of genus
  $g=2$ in Figure~\ref{fig:genus2}.  The mathematical software Sage~\cite{sage}
  was used for some of the calculations.
\begin{figure}[ht] 
\begin{tikzpicture}[scale=0.8]
\SetVertexMath
\GraphInit[vstyle=Art]
\SetUpVertex[MinSize=3pt]
\SetVertexLabel
\tikzset{VertexStyle/.style = {%
shape = circle,
shading = ball,
ball color = black,
inner sep = 1.5pt
}}
\SetUpEdge[color=black]
\Vertex[LabelOut,Lpos=90, Ldist=.1cm,x=0,y=1]{x}
\Vertex[LabelOut,Lpos=180, Ldist=.1cm,x=-0.866,y=0]{y}
\Vertex[LabelOut,Lpos=0, Ldist=.1cm,x=0.866,y=0]{z}
\Vertex[LabelOut,Lpos=270, Ldist=.1cm,x=0,y=-1]{s}
\Edges(x,y,s,z,x)
\Edge[](y)(z)
\end{tikzpicture}
\caption{Genus two graph $G$.}\label{fig:genus2}
\end{figure}
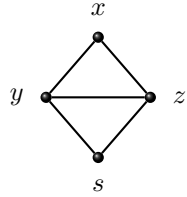
The sandpile group for $G$ is cyclic of order $8$.
Its toppling ideal is $I=(x^2-yz,y^3-xz,z^3-xy,yz-1)$, and its homogeneous
toppling ideal is 
\[
I_h=I^h=(x^2-yz,y^3-xzs,z^3-xys,yz-s^2,xz^2-y^2s,xy^2-z^2s).
\]
Letting $\omega=\exp(2\pi i/8)$, the zeros set of $I$ is
\[
Z(I)=\{( (-1)^j,\omega^{-j},\omega^{j}):0\leq j\leq 7\}\subset\C^3,
\]
which forms a cyclic group of order $8$ under component-wise multiplication.
With respect to the sandpile monomial ordering (\verb+grevlex+) for which
$x>y>z>s$, the normal basis for the coordinate ring of $G$ is the spanned by $8$
monomials:
\[
R/I=\C[x,y,z]/I=\myspan\{1,x,y,z,xy,xz,y^2,z^2\}.
\]
The exponent vectors of the normal basis give the superstable configurations:
\[
(0,0,0),(1,0,0),(0,1,0),(0,0,1),(1,1,0),(1,0,1),(0,2,0),(0,0,2),
\]
and dualizing, $c\to\cmax-c$, gives the recurrent configurations:
\[
(1,2,2),(0,2,2),(1,1,2),(1,2,1),(0,1,2),(0,2,1),(1,0,2),(1,2,0).
\]
(We use the notation $(c_1,c_2,c_3):=c_1\,x+c_2\,y+c_3\,z$.)

From the degrees of the monomials in the normal basis, one sees that the affine
Hilbert function for $G$ is
\[
H_G(0)=1,\quad H_G(1)=3,\quad H_G(2)=4
\]
with postulation number $2$ (equal to $g$, the degree of the maximal
superstables).  The Tutte polynomial for $G$ was calculated in
Figure~\ref{fig:tutte}, and in accordance with Corollary~\ref{cor:tutte}, the
Hilbert series for $G$ is
\[
y^2\,T_G(1,1/y)=1+3y+4y^2.
\]
The minimal free resolution for $G$ is
\[
\xymatrix{
0 & I \ar[l] \restore
& S^6 \ar[l]_-{\phi_0} \save -<0pt,44pt>*{\begin{array}{l}
\scriptstyle 0110 \\
\scriptstyle 2000 \\ 
\scriptstyle 0030 \\ 
\scriptstyle 1020 \\ 
\scriptstyle 0300 \\
\scriptstyle 1200 
\end{array}} \restore & S^9
\ar[l]_-{\phi_1} \save -<-3pt,49pt>*{\begin{array}{l} 
\scriptstyle 0121 \\
\scriptstyle 0211 \\
\scriptstyle 1201\times2 \\ 
\scriptstyle 1120 \\ 
\scriptstyle 1021\times 2\\
\scriptstyle 1210 \\
\scriptstyle 0220 
\end{array}}
\restore & S^4 \ar[l]_-{\phi_2}
\save -<0pt,32pt>*{\begin{array}{l} 
\scriptstyle 0122 \\
\scriptstyle 0212 \\
\scriptstyle 1022 \\
\scriptstyle 1202  
\end{array}}
\restore & 0 \ar[l] 
}
\]
The $\mbox{Cl}(G)$-degrees are listed in $x,y,z,s$ order.  The degrees of the
highest nonzero Betti numbers correspond to the minimal recurrent configurations
as prescribed by Theorem~\ref{thm:last betti number}.  For instance, the degree
$0122$ corresponds to the minimal alive divisor $y+2z+2s$ and to the minimal
recurrent configuration $(0,1,2)$.  Thus, $\beta_3=H_G(2)$, and the degrees of
each of these divisors is $5=\#E$.  

As an example of Hochster's formula (Theorem~\ref{betti homology}), let $D=1021=x+2z+s$.  The complete linear
system for $D$ is
\[
|D|=\{1021,2200,0202,0310\},
\]
and $\Delta_D$ is the simplical complex pictured in Figure~\ref{fig:hochster2}.
We have
\[
\beta_{2,1021}=\dim_{\C}\widetilde{H}_{1}(1021;\C)=2.
\]
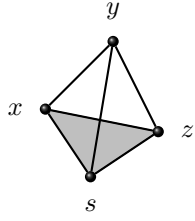
\begin{figure}[ht] 
\begin{tikzpicture}[scale=0.3]
\SetVertexMath
\GraphInit[vstyle=Art]
\SetUpVertex[MinSize=3pt]
\SetVertexLabel
\tikzset{VertexStyle/.style = {%
shape = circle,
shading = ball,
ball color = black,
inner sep = 1.5pt
}}
\SetUpEdge[color=black]
\draw[fill=gray!50] (0,0) -- (5,-1) -- (2,-3);
\Vertex[LabelOut,Lpos=180, Ldist=.1cm,x=0,y=0]{x}
\Vertex[LabelOut,Lpos=270, Ldist=.1cm,x=2,y=-3]{s}
\Vertex[LabelOut,Lpos=90, Ldist=.1cm,x=3,y=3]{y}
\Vertex[LabelOut,Lpos=0, Ldist=.1cm,x=5,y=-1]{z}
\Edges(x,z,s,x,y,s)
\Edge[](y)(z)
\end{tikzpicture}
\caption{The simplicial complex $\Delta_D$ for
Example~\ref{example:genus2}.}\label{fig:hochster2}
\end{figure}
\end{example}

\subsection{Conjecture.}
Let $G=(V,E,s)$ be an undirected sandpile graph.  For
$U\subseteq V$, let $G|_U$ denote the subgraph of $G$ induced by $U$, i.e.,
the graph with vertices $U$ and edges $e\in E$ such that both endpoints of $e$
are in $U$.  A {\em connected $k$-partition} or {\em $k$-bond} of $G$ is a partition
$\Pi=\sqcup_{i=1}^kV_i$ of $V$ such that $G|_{V_i}$ is connected for all~$i$.
The corresponding {\em $k$-partition graph}, $G_{\Pi}$, is the graph with
vertices $\{V_1,\dots,V_k\}$ and with edge weights
\[
\wt(V_i,V_j)=\#\{e\in E:\text{one endpoint of $e$ is in $V_i$ and the other is
in $V_j$}\}.
\]
We consider $G_{\Pi}$
to be a sandpile graph with sink vertex $V_i$, where $i$ is chosen so that $s\in
V_i$.

The following conjecture appears as Corollary~3.29 in~\cite{Wilmes}.  Using the
mathematical software Sage, it has been verified for all undirected, unweighted
graphs with fewer than $7$ vertices.
\begin{conj}\label{conj:wilmes} Let $\mathcal{P}_k$ denote the set of connected $k$-partitions of
  $G$.  Then
\[
\beta_k=\sum_{\Pi\in\mathcal{P}_{k+1}}\#\{c:\text{$c$ a minimal recurrent
configuration on $G_{\Pi}$}\}.
\]
\end{conj}
\begin{example}
  Figure~\ref{fig:conjecture} displays the $5$ connected $3$-partitions of $G$
  along with their corresponding $3$-partition graphs and $h$-vectors.  The top
  value of each $h$-vector is the number of minimal recurrent configurations (or
  maximal superstable configurations) on the partition graph.  Summing these top
  values gives $\beta_2$ for $G$.
\begin{figure}[ht] 
\begin{tikzpicture}[scale=0.7]
\SetVertexMath
\GraphInit[vstyle=Art]
\SetUpVertex[MinSize=3pt]
\SetVertexNoLabel
\tikzset{VertexStyle/.style = {%
shape = circle,
shading = ball,
ball color = black,
inner sep = 1.5pt
}}
\SetUpEdge[color=black]

\draw[shift={(-6.433,8.5)}, rotate=40.89] (0,0) ellipse(0.3 and 1.0);
\draw[shift={(-6,10)}] (0,0) circle(0.2);
\draw[shift={(-5.134,9)}] (0,0) circle(0.2);
\Vertex[x=-6,y=10]{x}
\Vertex[x=-6.866,y=9]{y}
\Vertex[x=-5.134,y=9]{z}
\Vertex[x=-6,y=8]{s}
\Edges(x,y,s,z,x)
\Edge[](y)(z)

\draw [ultra thick,style={->}] (-6,7.) -- (-6,6.0);

\Vertex[x=-6,y=4.4]{x}
\Vertex[x=-6.866,y=3.4]{y}
\Vertex[x=-5.134,y=3.4]{z}
\Edge[](x)(y)
\Edge[](x)(z)
\Edge[](y)(z)
\fill[color=white] (-6,3.4) circle(0.2cm);
\draw (-6,3.4) node{$2$};
\draw (-6,1.5) node{$1\ \ 2\ \ \mathbf{2}$};

\draw[shift={(-2.5673,8.5)}, rotate=-40.89] (0,0) ellipse(0.3 and 1.0);
\draw[shift={(-3,10)}] (0,0) circle(0.2);
\draw[shift={(-3.866,9)}] (0,0) circle(0.2);
\Vertex[x=-3,y=10]{x}
\Vertex[x=-3.866,y=9]{y}
\Vertex[x=-2.134,y=9]{z}
\Vertex[x=-3,y=8]{s}
\Edges(x,y,s,z,x)
\Edge[](y)(z)

\draw [ultra thick,style={->}] (-3,7.) -- (-3,6.0);

\Vertex[x=-3,y=4.4]{x}
\Vertex[x=-3.866,y=3.4]{y}
\Vertex[x=-2.134,y=3.4]{z}
\Edge[](x)(y)
\Edge[](x)(z)
\Edge[](y)(z)
\fill[color=white] (-3,3.4) circle(0.2cm);
\draw (-3,3.4) node{$2$};
\draw (-3,1.5) node{$1\ \ 2\ \ \mathbf{2}$};

\draw[shift={(0,9)}, rotate=90] (0,0) ellipse(0.3 and 1.2);
\draw[shift={(0,10)}] (0,0) circle(0.2);
\draw[shift={(0,8)}] (0,0) circle(0.2);
\Vertex[x=0,y=10]{x}
\Vertex[x=-0.866,y=9]{y}
\Vertex[x=0.866,y=9]{z}
\Vertex[x=0,y=8]{s}
\Edges(x,y,s,z,x)
\Edge[](y)(z)

\draw [ultra thick,style={->}] (0,7.0) -- (0,6.0);

\Vertex[x=0,y=5.2]{x}
\Vertex[x=0,y=4]{y}
\Vertex[x=0,y=2.8]{z}
\Edges(x,y,z)
\fill[color=white] (0,4.6) circle(0.25cm);
\draw (0,4.6) node{$2$};
\fill[color=white] (0,3.4) circle(0.25cm);
\draw (0,3.4) node{$2$};
\draw (0,1.5) node{$1\ \ 2\ \ \mathbf{1}$};

\draw[shift={(2.567,9.5)}, rotate=-40.89] (0,0) ellipse(0.3 and 1.0);
\draw[shift={(3,8)}] (0,0) circle(0.2);
\draw[shift={(3.866,9)}] (0,0) circle(0.2);
\Vertex[x=3,y=10]{x}
\Vertex[x=2.134,y=9]{y}
\Vertex[x=3.866,y=9]{z}
\Vertex[x=3,y=8]{s}
\Edges(x,y,s,z,x)
\Edge[](y)(z)

\draw [ultra thick,style={->}] (3,7.0) -- (3,6.0);

\Vertex[x=3,y=3.4]{x}
\Vertex[x=2.134,y=4.4]{y}
\Vertex[x=3.866,y=4.4]{z}
\Edge[](x)(y)
\Edge[](x)(z)
\Edge[](y)(z)
\fill[color=white] (3,4.4) circle(0.2cm);
\draw (3,4.4) node{$2$};
\draw (3,1.5) node{$1\ \ 2\ \ \mathbf{2}$};

\draw[shift={(6.433,9.5)}, rotate=40.89] (0,0) ellipse(0.3 and 1.0);
\draw[shift={(5.134,9)}] (0,0) circle(0.2);
\draw[shift={(6,8)}] (0,0) circle(0.2);
\Vertex[x=6,y=10]{x}
\Vertex[x=5.134,y=9]{y}
\Vertex[x=6.866,y=9]{z}
\Vertex[x=6,y=8]{s}
\Edges(x,y,s,z,x)
\Edge[](y)(z)

\draw [ultra thick,style={->}] (6,7.0) -- (6,6.0);

\Vertex[x=6,y=3.4]{x}
\Vertex[x=5.134,y=4.4]{y}
\Vertex[x=6.866,y=4.4]{z}
\Edge[](x)(y)
\Edge[](x)(z)
\Edge[](y)(z)
\fill[color=white] (6,4.4) circle(0.2cm);
\draw (6,4.4) node{$2$};
\draw (6,1.5) node{$1\ \ 2\ \ \mathbf{2}$};

\draw (-9,9.3) node{connected};
\draw (-9,8.6) node{$3$-partitions};
\draw (-9,4) node{$3$-partition};
\draw (-9,3.5) node{graphs};
\draw (-9,1.5) node{$h$-vector};
\end{tikzpicture}
\caption{Second Betti number: $\beta_2=2+2+1+2+2=9$.}\label{fig:conjecture}
\end{figure}
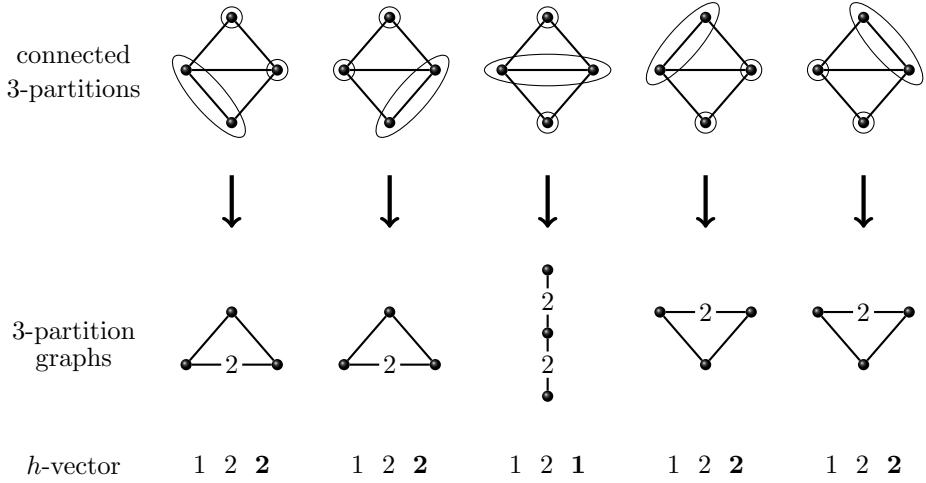
\end{example}
As a corollary to Conjecture~\ref{conj:wilmes}, it is shown in~\cite{Wilmes}
that
\begin{cor}\label{cor:wilmes} If Conjecture~\ref{conj:wilmes} is true, then the following five
  statements are also true.
  \begin{enumerate}
    \item The number of polynomials, $\beta_1$, in a minimal generating set for the
      homogeneous toppling ideal of $G$ is equal to the number of {\em cuts}
      (i.e., the number of connected $2$-partitions) of $G$.
    \item For a tree on $n$ vertices, $\beta_k={{n-1}\choose{k}}$.
    \item If the weight of an edge of $G$ is changed from one nonzero value to
      another, the $\beta_k$ do not change.
    \item If $G'$ is obtained from $G$ by adding an edge to $G$ (between two
      vertices of $G$), then $\beta_k(G)\leq\beta_k(G')$ for all $k$.
    \item\label{cor:wilmes5} For the complete graph on $n$ vertices, $K_n$, we
      have that $\beta_k$ is the number of strictly ascending chains of length
      $k$ of nonempty subsets of $[n-1]:=\{1,\dots,n-1\}$.
  \end{enumerate}
  \begin{remark}
    Corollary~\ref{cor:wilmes}~(\ref{cor:wilmes5}) is
    proved, independently, in~\cite{Madhu}.
  \end{remark}
\end{cor}
\section{Gorenstein toppling ideals}\label{section:gor}
This section characterizes toppling ideals that are complete intersection
ideals and gives a method for constructing Gorenstein toppling ideals.  
\subsection{Complete intersections}
If $V\subset\mathbb{P}^n$ is the solution set to a system of homogeneous
polynomials, then $V$ is a {\em complete intersection} if the ideal generated by
all homogeneous polynomials vanishing on $V$ can be generated by a set of
polynomials with cardinality equal to the codimension of $V$ in $\mathbb{P}^n$.
Specializing to the case of sandpiles, we get the following definition.
\begin{definition}
  Let $G=(V,E,s)$ be a sandpile graph with homogeneous toppling ideal $I$.
  Then $G$ is a {\em complete intersection} sandpile graph if $I$ is generated by
  $|V|-1$ homogeneous polynomials.  (We also say that $I$ or the set of zeros of $I$ is a
  complete intersection.)

  Let $\Lap$ be a submodule of $\Z^{n+1}$ of rank $n$ whose lattice ideal
  $I(\Lap)$ is homogeneous.  Then $I(\Lap)$ is a {\em complete intersection} if
  it is generated by $n$ homogeneous polynomials.
\end{definition}
\begin{remark}\label{homogeneous lattice}
  The lattice ideal $I(\Lap)$ is generated by homogeneous polynomials if and
  only if $\deg(w):=\sum_iw_i=0$ for all $w\in\Lap$.
\end{remark}

For the following, recall from \S\ref{sandpiles} that sandpile has an {\em
absolute} sink if its sink has outdegree $0$.
\begin{definition}
For $i=1,2$, let $G_i=(V_i,E_i,s_i)$ be a sandpile graph with 
edge-weight function $\mathrm{wt}_i$ and absolute sink $s_i$.  Suppose that
the two graphs are vertex-disjoint.  Let $G$ be any graph with vertex set
$V = V_1 \sqcup V_2$, and edge-weight function,~$\wt$, satisfying the
following
\begin{enumerate}
  \item $\wt(e)=\wt_1(e)$ if $e\in E_1$,
  \item $\wt(e)=\wt_2(e)$ if $e\in E_2$,
  \item $\wt(u,v)=0$ if $(u,v)\in (\tV_1\times V_2)\cup(V_2\times V_1)$,
  \item $\wt(s_1,v)>0$ for some $v\in V_2$.
\end{enumerate}
We consider $G$ to be a sandpile graph with $s_2$ as its absolute sink.
Let $\Delta:=\Delta_G$ be the Laplacian of $G$, and define
\[
D:=\Delta(s_1)|_{V_1}=\sum_{v\in V_1}\Delta(s_1)_{v}\,v,
\]
a divisor on $G_1$.  Then $G$ is a {\em wiring of $G_1$ into $G_2$} with {\em
wiring divisor $D$} if $|D|\neq\emptyset$, i.e., if the complete linear system
for $D$ as a divisor on $G_1$ is nonempty (cf.~\S\ref{section:resolutions}).
\end{definition}
Thus, to form a wiring of $G_1$ into $G_2$, one connects $s_1$ into $G_2$ with
at least one edge and then adds edges from $s_1$ back into~$G_1$ as
determined by a divisor,~$D$, on~$G_1$ having a nonempty complete linear system.
There always exists some wiring of $G_1$ into $G_2$.  For instance, we could
take $D=k\,s_1$ for any $k>0$ by connecting~$G_1$ to $G_2$ with $k$ edges
from $s_1$ into $G_2$ (and no edges from $s_1$ back into $G_1$).
\begin{figure}[ht] 
  \begin{center}
  \includegraphics{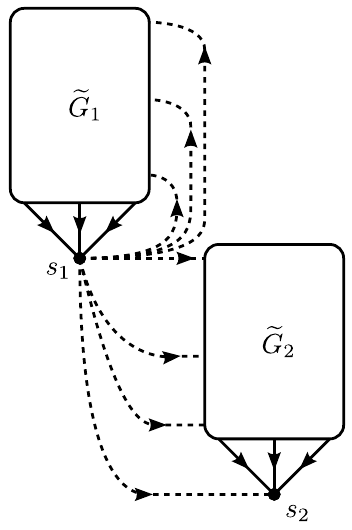}
  \end{center}
\caption{A wiring of $G_1$ into $G_2$.}\label{fig:wiring}
\end{figure}
\begin{notation}
  For any sandpile graph $G=(V,E,s)$,  with Laplacian $\Delta_G$, we let
  $\Delta^{\circ}_G=\Delta_G|_{\tV}$.  Thus, $\Delta^{\circ}_{G}:\Z\tV\to\Z V$,
  and in terms of matrices, $\Delta^{\circ}_{G}$ is formed from~$\Delta_{G}$ by
  removing the column corresponding to the sink---a column of zeros if~$G$ has
  an absolute sink.  We will call $\Delta^{\circ}_G$ the {\em restricted
  Laplacian} of $G$.
\end{notation}
With this notation, if $G$ is a wiring of $G_1$ into $G_2$, then
\[
\Delta^{\circ}_G=
\left(
\begin{array}{ccc}
  \Delta^{\circ}_{G_1}&0&\alpha\\
  0&\Delta^{\circ}_{G_2}&\beta
\end{array}
\right)
\]
where exactly one entry of $\alpha$ is positive (corresponding to $s_1$) and
$\beta\le0$.  The last column corresponds to $s_1$, and the wiring
divisor is $D=\alpha$.

If $G_1$ is a single point with no edges, then we
regard $\Delta^{\circ}_{G_1}$ as the $1\times 0$ empty matrix, and $\alpha$ will
be a single integer, as in the following example.

\begin{example}\label{example:wiring}
  Let $G_1$ be the graph with a single vertex $s_1$ and no edges.  Let $G_2$
  have vertex set $\{v_2,v_3,s_2\}$ and edge set $\{(v_2,s_2),(v_3,s_2)\}$.
  Figure~\ref{fig:dag} illustrates a wiring, $G$, of $G_1$ into $G_2$.
\begin{figure}[ht] 
\begin{tikzpicture}[scale=1.0]
\SetVertexMath
\GraphInit[vstyle=Art]
\SetUpVertex[MinSize=3pt]
\SetVertexLabel
\tikzset{VertexStyle/.style = {%
shape = circle,
shading = ball,
ball color = black,
inner sep = 1.5pt
}}
\SetUpEdge[color=black]
\Vertex[LabelOut,Lpos=90, Ldist=.1cm,x=0,y=1]{s_1}
\Vertex[LabelOut,Lpos=180, Ldist=.1cm,x=-0.866,y=0]{v_2}
\Vertex[LabelOut,Lpos=0, Ldist=.1cm,x=0.866,y=0]{v_3}
\Vertex[LabelOut,Lpos=270, Ldist=.1cm,x=0,y=-1]{s_2}
\Edge[style={-triangle 45}](s_1)(v_2)
\Edge[style={-triangle 45}](s_1)(v_3)
\Edge[style={-triangle 45}](v_2)(s_2)
\Edge[style={-triangle 45}](v_3)(s_2)
\end{tikzpicture}
\caption{The wiring $G$ for Example~\ref{example:wiring}.}
\label{fig:dag}
\end{figure}
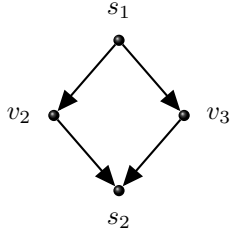
The wiring divisor is $D=2s_1$.  The restricted
Laplacian of $G$ is, with respect to the indicated vertex ordering,
\[
\Delta^{\circ}_G=
\bordermatrix{
&\hfill v_2&\hfill v_3&\hfill s_1\cr
s_1&\hfill 0&\hfill 0&\hfill 2\cr
v_2&\hfill 1&\hfill 0&-1\cr
v_3&\hfill 0&\hfill 1&-1\cr
s_2&-1&-1&\hfill 0
}.
\]
\end{example}

\begin{definition}
A directed multigraph $G$ is \emph{completely wired} if it is a single
vertex with no edges or
if it is the wiring of one completely wired graph into another.
\end{definition}

\begin{example} Every directed acyclic graph is completely wired.
\end{example}

\begin{definition}
An integral matrix is \emph{mixed} if each column contains both positive and
negative entries.
An integral matrix is \emph{mixed dominating} if it does not contain a mixed square
submatrix.
\end{definition}

Empty $d\times 0$ matrices are mixed dominating by convention.
The following two theorems are established in \cite{MT} and \cite{FMS}.
\begin{thm}\label{theorem:ci-mixeddom}
Let $\mathcal{L}$ be a submodule of $\Z^{n+1}$ of rank $n$ such that the
associated lattice ideal $I(\mathcal{L})$ is homogeneous. Then $I(\mathcal{L})$
is a complete intersection if and only if there exists a basis $u_1, \ldots,
u_n$ for $\mathcal{L}$ such that the matrix whose columns are the~$u_i$ is mixed
dominating.
\end{thm}

\begin{thm}\label{theorem:mixeddom}
If $M$ is a mixed dominating matrix, then by reordering its columns and rows we
may obtain
\[
M' = \left( \begin{tabular}{cc|c} $M_1$ & $0$ & $\alpha$ \\ $0$ & $M_2$ &
$\beta$ \end{tabular} \right),
\]
where the $M_i$ are mixed dominating, $\alpha\ge 0$, and $\beta\le 0$.
\end{thm}

It is allowable for the matrix $M_1$ in Theorem~\ref{theorem:mixeddom} to be the
empty $d \times 0$ matrix, in which case we would have
\[
M' = \left( \begin{array}{c|c} 0 & \alpha \\ 
  M_2 & \beta \end{array} \right),
\]
where the upper-left block is a zero matrix with $d$ rows.  A similar statement
holds if $M_2$ is the $d\times0$ matrix, in which case we would have a
lower-left zero matrix block.

We now characterize complete intersection sandpile graphs.

\begin{thm}\label{ci implies wired}
Let $\mathcal{L}$ be a submodule of $\Z^{n+1}$ of rank $n$ such that the
associated lattice ideal $I(\mathcal{L})$ is a complete intersection. Then there
exists a completely wired graph $G$ whose Laplacian lattice is
$\mathcal{L}$, and hence, $I(\mathcal{L}) = I(G)^h$, where $I(G)^h$ is the
homogeneous toppling ideal of~$G$.
\end{thm}
\begin{proof}
We proceed by induction, the case $n=0$ being trivial.  Let $u_1, \ldots,
u_n$ be a basis for $\mathcal{L}$, and let $M$ be the matrix whose columns are
the $u_i$. By Remark~\ref{homogeneous lattice}, we have $\deg(u_i)=0$ for all
$i$. (Here, $\deg(u_i)$ denotes the degree of $u_i$ as a divisor, i.e., the sum
of the components of $u_i$.)  By Theorems
\ref{theorem:ci-mixeddom} and \ref{theorem:mixeddom}, we may assume that
\[
M = \left( \begin{tabular}{cc|c} $M_1$ & $0$ & $\alpha$ \\ $0$ & $M_2$ &
$\beta$ \end{tabular} \right)
\]
where the $M_i$ are mixed dominating, $\alpha\ge0$, and $\beta\le0$.  Each
column of $M_1$ and~$M_2$ has entries that sum to zero.  By our rank assumption,
it follows that $M_1$ and $M_2$ are matrices of full rank, each with one more
row than column.  By induction, there exist completely wired graphs $G_1$ and
$G_2$ such that $\im(\Delta_{G_i}) = \im(M_i)$ for $i=1,2$. Let~$s_1$ be the sink
of $G_1$.  Let $c$ be any nonnegative configuration on~$G_1$ with full support
and contained in $\im(\tD_{G_1})$, the reduced Laplacian lattice for $G_1$.  For instance,
we could take $c=\delta-\delta^{\circ}$ where
$\delta=\sum_{v\in\tV_1}(\outdeg(v)+1)\,v$.  Define the divisor 
$D=c-\deg(c)\,s_1\in\im(\Delta_{G_1})=\im(\Delta^{\circ}_{G_1})$.  Take
$k\in\N$ such that $k\cdot c +\deg(\alpha)\,s_1
\ge \alpha$. Now
\[
M'=
\left(
\begin{tabular}{cc|c} 
  $\Delta^{\circ}_{G_1}$ & $0$ & $\alpha - kD$\\
  $0$ & $\Delta^{\circ}_{G_2}$ & $\beta$
\end{tabular} \right)
\]
has the same column span as $M$, and $M'=\Delta^{\circ}_{G}$ where $G$ is
the wiring of $G_1$ into~$G_2$ with wiring divisor $\alpha - kD$.
Then $G$ is completely wired and, up to an ordering of its vertices, its
full Laplacian lattice is $\Lap$.
\end{proof}
\begin{example}
  The graph of Example~\ref{example:bad sink} is a complete intersection
  sandpile graph.  It is not completely wired, but its Laplacian lattice is the
  same as that for the completely wired graph consisting of a single directed edge
  connecting $v_1$ to $v_2$.
\end{example}
\begin{thm}\label{wired implies ci}
If the graph $G$ is completely wired, then $I(G)^h$ is a complete
intersection.
\end{thm}
\begin{proof}
If $G$ has only one vertex, then $I(G) = \{0\}$ is a complete
intersection, so we will again proceed by induction, now on $|V(G)|$.
Assume $|V(G)|>1$ and that~$G$ is the wiring of some graph
$G_1$ with sink $s$ into another graph $G_2$ with wiring divisor~$D$.
Let $\beta = \Delta_{G}(s)|_{V_2}$. Then 
\[
\Delta^{\circ}_{G}= 
\left(
\begin{tabular}{cc|c} 
  $\Delta^{\circ}_{G_1}$ & $0$ & $D$\\
  $0$ & $\Delta^{\circ}_{G_2}$ & $\beta$
\end{tabular} \right).
\]
By Theorem \ref{theorem:ci-mixeddom} and induction, there exist
$M_1$ and $M_2$ with $\im(M_i) = \im(\Delta_{G_i})$ for $i=1,2$, and $E \in
|D|$, such that
\[
M = \left( \begin{tabular}{cc|c} $M_1$ & $0$ & $E$ \\ $0$ & $M_2$ &
$\beta$ \end{tabular} \right)
\]
has the same column span as $\Delta^{\circ}_{G}$ and is mixed dominating. So
$I(G)^h$ is a complete intersection by Theorem \ref{theorem:ci-mixeddom}.
\end{proof}

\subsection{Gorenstein sandpile graphs}
Having characterized complete intersection sandpile graphs, we proceed to give
a method for constructing sandpile graphs with Gorenstein toppling ideals.
Our basic reference for Gorenstein ideals is~\cite{Geramita}.

\begin{notation}
  Let $S=\C[x_1,\dots,x_{n+1}]$, and let $I$ be a homogeneous ideal in $S$.  Let $S_d$
be the $\C$-vector space generated by all homogeneous polynomials of degree~$d$,
and let $I_d:=I\cap S_d$. Define $A=S/I$, and let $A_d:=(S/I)_d:=S_d/I_d$.  
Let 
\[
\mathfrak{m}=(x_1,\dots,x_{n+1})
\]
denote the unique maximal homogeneous ideal in either $S$ or in $A$.
\end{notation}

\begin{definition}
  The {\em socle} of $A$ is
  \[
  \mathrm{Soc}(A):=(0:\mathfrak{m}):=\{f\in A:f\,\mathfrak{m}=0\}.
  \]
\end{definition}

\begin{definition}
  The ring $A$ is {Artinian} if $\dim_{\C}A<\infty$.  In that case, we write
  \[
  A = \C\oplus A_1\oplus\dots\oplus A_{\ell},
  \]
  with $A_{\ell}\neq0$.  The number $\ell$ is the {\em socle degree} of
  $A$.  It is the least number $\ell$ such that $\mathfrak{m}^{\ell+1}\subseteq
  I$.
\end{definition}

\begin{definition}
  The ring $A$ is {\em Gorenstein} if it is Artinian and
  $\dim_{\C}\mathrm{Soc}(A)=1$ (so $\mathrm{Soc}(A)=A_{\ell}$ and
  $\dim_{\C}A_{\ell}=1$).
\end{definition}

\begin{prop} \label{gor-pairing} Suppose $A$ is Artinian with socle degree
$\ell$.  Then $A$ is Gorenstein if and only if $\dim_{\C}A_{\ell}=1$ and the
pairing given by multiplication
\[
A_d\times A_{\ell-d}\to A_{\ell}\approx \C
\]
is a perfect pairing.
\end{prop}
\begin{proof}
See the proof of, and remarks following, Proposition 8.6, \cite{Geramita}.
\end{proof}

As an easy corollary, we have 
\begin{cor}
  The Hilbert function of an Artinian Gorenstein ring $A$ is symmetric.   That is,
  if the socle degree of $A$ is $\ell$, then
\[
H_A(d)=H_A(\ell-d)
\]
for all $d$.
\end{cor}

Now let $S'=\C[y_1,\dots,y_{n+1}]$, and let $S$ act on $S'$ by treating
each $x_i$ as the differential operator $\partial/\partial y_i$.
\begin{thm} (Macaulay, cf.~Theorem 8.7~\cite{Geramita})
  The ring $A=S/I$ is Gorenstein with socle degree $\ell$ if and only if there
  exists a nonzero $g\in S'_{\ell}$ such that
  \[
  I=\mathrm{ann}(g):=\{f\in S: f(\partial/\partial y_1,\dots,\partial/\partial
  y_{n+1})\,g=0\}.
  \]
\end{thm}

Now consider the case where $I$ is the homogeneous toppling ideal for a sandpile
graph $G$ with vertices $\{v_1,\dots,v_{n+1}\}$.  Let $X=Z(I)$ be the zero set of $I$ as discussed in
section~\ref{projective points}. Let $a\in S$ be a linear polynomial that does not vanish at any
point of $X$.  For instance, $a$ may be any of the indeterminates, $x_i$.  Restricting the exact
sequence given by multiplication by $a$,
\[
0\to A\stackrel{\cdot a}{\longrightarrow} A\to A/(a)\to 0,
\]
to each degree $d$, we find that the Hilbert function for $A/(a)$ is the first
differences of the Hilbert function for $A$, i.e., $H_{A/(a)}(d)=\Delta H_A(d)$.
It then follows from~(\ref{eqn:first differences}) that $A/(a)$ is Artinian.
\begin{definition}\label{reduction}
  Continuing the notation from above, 
  the ring $A/(a)$ is called an {\em Artinian reduction} of $A$.
  Let $\ell$ be the socle degree of an Artinian reduction of~$A$,
  and let $h_d:=\Delta H_A(d)$ for $d=0,\dots,\ell$.  Then
  $(h_0,\dots,h_{\ell})$ is the {\em homogeneous $h$-vector} of $G$ (or $I$ or $X$).
\end{definition}
\begin{remark}
The homogeneous $h$-vector and the $h$-vector appearing in
Definition~\ref{def:h-vector} are identical in the case the $\Delta(v_{n+1})$
is in the span of $\{\Delta(v_i):1\leq i\leq n\}$ (see the discusion after
Example~\ref{example:h-vector}).
\end{remark}
\begin{definition}
We say $G$ is a {\em Gorenstein sandpile graph} if its homogeneous coordinate
ring has a Gorenstein Artinian reduction.  We also say that $I$  and $X$ are
(arithmetically) Gorenstein.  
\end{definition}

\begin{remark} \ 
 \begin{enumerate}
   \item Using the notation preceding Definition~\ref{reduction}, it turns out that if $A$ has a Gorenstein
     Artinian reduction, then every Artinian reduction of $A$ is Gorenstein.
   \item The notion of a Gorenstein ideal is much more general, but requires a
     discussion of the Cohen-Macaulay property, which our toppling ideals
     (defining a finite set of projective points) satisfy automatically
     (cf.~\cite{Eisenbud}).
 \end{enumerate}
\end{remark}
It is well-known that complete intersection ideals are Gorenstein (cf.~\S
21.8\cite{Eisenbud}).  In particular, we have the following.
\begin{thm}\label{thm:ci implies gor}
  Let $G$ be a sandpile graph.  If $G$ is a complete intersection, then~$G$ is
  Gorenstein.
\end{thm}

\begin{thm}\label{thm:gor}
  Let $I$ be the homogeneous toppling ideal of the sandpile graph $G$
  having $n+1$ vertices.  The following are equivalent:
  
 \begin{enumerate}
   \item\label{gor-one} $G$ is Gorenstein;
   \item\label{gor-two} if the minimal free resolution for $I$ is
  \[
 0\leftarrow I
 \stackrel{\phi_0}{\longleftarrow}
 F_0 
 \stackrel{\phi_1}{\longleftarrow}
 F_1
\leftarrow
\dots
 \stackrel{\phi_n}{\longleftarrow}
 F_n
  \leftarrow
  0, 
  \]
  then $F_n\approx S$ as an $S$-module;
\item\label{gor-three} the homogeneous $h$-vector for $G$ is
    symmetric.
\end{enumerate}
\end{thm}
\begin{proof}
  The equivalence of items (\ref{gor-one}) and (\ref{gor-two}) is a standard
  result (cf.~\cite{Eisenbud}). The equivalence of items (\ref{gor-one}) and (\ref{gor-three}) follows
  by \cite{Geramita-Cayley} since $I$ is a Cayley-Bacharach ideal by
  Proposition~\ref{prop:CB-yes}.
\end{proof}
\begin{example}
Let $G$ be as in example \ref{ex:resolution}. We saw that the last nonzero
module in the free
resolution for $I(G)^h$ is $S(-(1,0,2,2))$, which is isomorphic to $S$ as
an $S$-module. Thus, the caption for Figure \ref{fig:gor}, stating that $G$
is Gorenstein, is justified by (\ref{gor-two}) above.
\end{example}

Define a {\em loopy tree} to be a (finite) graph that is formed from a weighted,
undirected tree by adding weighted loops at some (maybe none) of the vertices.
\begin{thm}\label{thm:loopy}
  For an undirected sandpile graph $G$, the following are equivalent:
  \begin{enumerate}
    \item $G$ is a loopy tree;
    \item $G$ is a complete intersection;
    \item $G$ is Gorenstein.
  \end{enumerate}
\end{thm}
\begin{proof}
Let $G=(V,E,s)$ be a undirected sandpile graph.  First suppose
that $G$ is a loopy tree.  Removing any outgoing edges from $s$ leaves a
completely wired graph having the same homogeneous toppling ideal as $G$.
Hence, $G$ is a complete intersection by Theorem~\ref{wired implies ci}, and
hence $G$ is Gorenstein by Theorem~\ref{thm:ci implies gor}.

We now assume that $G$ is not a loopy tree.  Since the lattice ideal of $G$ is
not affected by loops, for ease of exposition we assume that $G$ has no loops.
By Theorem~\ref{thm:last betti number} and
Theorem~\ref{thm:gor}~(\ref{gor-two}), we have that $G$ is Gorenstein if and
only if it has a unique minimal recurrent configuration.  

To characterize the minimal recurrent configurations, let $\prec$ be a total
ordering of the vertices such that for all nonsink vertices $v$, (i) $s\prec v$,
and (ii) there exists $u\prec v$ such that $\{u,v\}\in E$.  Define the
configuration $c_{\prec}$ by
\[
c_{\prec,v}:=\deg(v)-\#\{v\in V:\text{$\{u,v\}\in E$ and $u\prec v$}\}.
\]
We now invoke Dhar's burning algorithm.  Let $b$ be the minimal burning
configuration for $G$.  By Theorem~\ref{thm:speer existence} it has script
$\vec{1}$, and by Theorem~\ref{thm:speer}, a configuration $c$ is recurrent if
and only if each nonsink vertex fires in the stabilization of $b+c$.  Note that
$b+c$ is obtained by starting with $c$ and firing the sink vertex.  It follows
that $c_{\prec}$ is a minimal recurrent configuration and that all minimal
recurrent configurations arise as $c_{\prec}$ for some ordering
$\prec$ satisfying (i) and (ii), above.

Let $C$ be a (undirected) cycle in $G$.  Choose a path $P$ in $G$ starting at
$s$ and going to a vertex of $C$, then traveling around $C$.  To be precise, let
$u_1,\dots, u_i$ be distinct vertices forming a path in $G$ (so
$\{u_{\ell},u_{\ell+1}\}\in V$ for all $\ell$) with $u_1=s$ and~$u_i$ a vertex in
$C$.  Assume that $u_i$ is the first vertex in the path to be in $C$.  (If~$s$
is in~$C$, then $i=1$.)  Next, let $u_{i},\dots,u_{i+j}$ be the vertices in the
cycle $C$, in order.  Then $P$ is the path $u_1,\dots,u_{i+j}$.  Let $\prec_1$ be
any total ordering satisfying (i) and (ii), above, with
\[
u_1\prec_1\cdots\prec_1 u_{i+j},
\]
and such that $u_k\prec_1v$ for all $u_k$ and all vertices $v$ not in $P$.  Let $\prec_2$ be
any total order satisfying (i) and (ii) with 
\[
u_1\prec_2\cdots\prec_2u_i\prec_2 u_{i+j}\prec_2u_{i+1}\prec_2u_{i+2}\prec_2\dots\prec_2
u_{i+j-1},
\]
and such that $u_k\prec_2v$ for all $u_k$ and all vertices $v$ not in $P$.  It follows that~$c_{\prec_1}$
and~$c_{\prec_2}$ are distinct minimal recurrent configurations on
$G$.  Hence, $G$ is not Gorenstein.
\end{proof}

By Theorem \ref{thm:normal basis}, an Artinian reduction of $A$ for a sandpile
graph with absolute sink has the set
\[
\{ x^c : c \text{ is a superstable configuration of $G$}\}
\]
as a normal basis. It follows that the socle degree $\ell$ of $A$ is the maximum
of the degrees of the superstable configurations of $G$. Hence, by
Theorem~\ref{thm:gor}~(\ref{gor-three}), a sandpile graph with absolute sink is
Gorenstein if and only if there exists a bijection between the superstable
configurations of degree $k$ and those of degree $\ell - k$.

\begin{lemma} \label{lemma:wirefromtrivial} Let $G_1$ be the graph on a
single vertex $v$ and let $G_2$ be a Gorenstein sandpile graph. Let $G$
be a wiring of $G_1$ into $G_2$. Then $G$ is Gorenstein.
\end{lemma}
\begin{proof}
  Let $\mathcal{A}$ be the set of superstable configurations on $G_2$ and define
  the integer $\ell:=\max\{\deg(a) : a \in \mathcal{A}\}$. Let $f\colon
  \mathcal{A}\to \mathcal{A}$ be a bijection such that $\deg(f(a)) = \ell -
  \deg(a)$ for all $a\in \mathcal{A}$. Let $d:=\outdeg(v)$.  Since there are no
  edges from vertices of $G_2$ to $v$ in $G$, the set of superstable
  configurations on $G$ is
\[
\mathcal{B}:=\{kv+a:\text{$a\in \mathcal{A}$ and $0\leq k<d$}\}.
\]
Let $m:=\max\{\deg(b):b\in \mathcal{B}\}=\ell+d-1$.
Define $g\colon \mathcal{B}\to\mathcal{B}$ by $g(kv + a) = (d-1-k)v+f(a)$ where $a\in
\mathcal{A}$. Then $g$ is a bijection and 
\[
\deg(g(kv + a)) =
\deg((d-1-k)v+f(a)) = \ell + d-1 - \deg(kv + a).
\]
It follows that $G$ is Gorenstein.
\end{proof}

\begin{lemma} \label{lemma:wireintotrivial} Let $G_1$ be a Gorenstein
sandpile graph with absolute sink $s$ and let~$G$ be a wiring of $G_1$ into
the graph on a single vertex $v$ with no edges. Then $G$ is Gorenstein. 
\end{lemma}
\begin{proof}
Let $\Delta = \Delta_G$ be the Laplacian matrix for $G$, and let $D$
be the wiring divisor of $G$. If $d$ is the weight of the edge from $s$ to
$v$, then
\[
\Delta^{\circ}_{G}= 
\left(\begin{array}{cr}
  \Delta^{\circ}_{G_1} & D\\
  0&-d 
\end{array}\right).
\]
Since $|D|\neq\emptyset$ by the definition of a wiring, there exists some
effective divisor $E \sim_{G_1} D$.
Thus, we can replace the last column of $\Delta^{\circ}_{G}$ with
\[
\left(\begin{array}{r} 
  E\\
  -d
\end{array}\right)
\]
without changing the column span, and hence without changing the associated
lattice ideal. Negating this column and swapping rows, the matrix
$\Delta^{\circ}_{G}$ becomes 
\[
\Delta^{\circ} := 
\left(\begin{array}{cc} 
  0&d\\ 
  \Delta^{\circ}_{G_1}&-E
\end{array}\right),
\]
which is the restricted Laplacian for a wiring of vertex
$v$ into $G_1$.  This graph is Gorenstein by Lemma \ref{lemma:wirefromtrivial}.
\end{proof}

\begin{thm}\label{theorem:wiregorenstein} Let $G_1$ and $G_2$ be Gorenstein
  sandpile graphs with absolute sinks. If $G$
  is a wiring of $G_1$ into $G_2$, then $G$ is Gorenstein.
\end{thm}
\begin{proof}
Let $D$ be the wiring divisor of $G$. Let $G'$ be the wiring of $G_1$ into the
graph on a single vertex $s$, disjoint from the vertices of $G_1$ or of $G_2$,
with wiring divisor~$D$. Let $\mathcal{A}'$ be the set of superstable
configurations on $G'$ and define the integer $\ell':=\max\{\deg(c):c\in
\mathcal{A}'\}$.  Since $G'$ is Gorenstein by Lemma \ref{lemma:wireintotrivial},
there exists a bijection $f'\colon \mathcal{A}'\to\mathcal{A}'$ such that
$\deg(f'(c))=\ell'-\deg(c)$. Let $\mathcal{A}_2$ be the set of superstables on
$G_2$, let $\ell_2=\max\{\deg(c):c\in \mathcal{A}_2\}$, and let $f_2\colon
\mathcal{A}_2\to \mathcal{A}_2$ be a bijection such that
$\deg(f_2(c))=\ell_2-\deg(c)$.

Clearly, if $c$ is superstable on $G$, then $c|_{\tV_2} \in \mathcal{A}_2$, and
$c|_{V_1} \in \mathcal{A}'$. Conversely, if~$c'\in \mathcal{A}'$ and $c_2 \in
\mathcal{A}_2$, then the configuration $c'+c_2$ is superstable on $G$.  Let
$\mathcal{A}=\{c' + c_2 : c' \in \mathcal{A}', c_2\in \mathcal{A}_2\}$, so that
$\mathcal{A}$ is the set of superstable configurations on $G$, and
$\max\{\deg(c) : c \in \mathcal{A}\} = \ell' + \ell_2=: \ell$. Define the
function $f\colon \mathcal{A}\to \mathcal{A}$ by $f(c' + c_2) = f'(c') +
f_2(c_2)$, where $c'\in \mathcal{A}'$ and $c_2\in \mathcal{A}_2$. Then $f$ is a
bijection, and $\deg(f(c' + c_2)) = \ell - \deg(c' + c_2)$.  Hence, $G$ is
Gorenstein.
\end{proof}

\bibliography{primer}{}
\bibliographystyle{amsplain}

\end{document}